\documentclass[a4paper]{amsart}
\usepackage[all]{xy}
\usepackage{amsmath}
\usepackage{amsfonts,amssymb}

\numberwithin{equation}{section}
\newtheorem{satz}{Satz}
\newtheorem{proposition}[satz]{Proposition}
\newtheorem{theorem}{Theorem}
\newtheorem*{theorem*}{Theorem}
\newtheorem{definition}[satz]{Definition}
\newtheorem{lemma}[satz]{Lemma}

\newtheorem{corollary}[satz]{Corollary}

\theoremstyle{definition}
\newtheorem*{remark*}{Remark}
\newtheorem{remark}[satz]{Remark}
\newtheorem*{notation}{Notation}
\newtheorem{claim}[satz]{Claim}
\newtheorem*{example*}{Examples}
\newtheorem{assumption}{Assumption}

\newcommand{\tensor}{\otimes}

\newcommand{\map}[1]{\stackrel{#1}{\longrightarrow}}
\newcommand{\incl}[1]{\stackrel{#1}{\hookrightarrow}}

\newcommand{\mat}[4]{
 \left(  \begin{matrix} #1 & #2 \\ #3 & #4 \end{matrix} \right)}
\newcommand{\dmat}[9]{\left( \begin{matrix} #1 & #2 & #3 \\ #4 & #5 & #6\\ #7 & #8& #9 \end{matrix}\right)}

\def\kbar{{\overline{k}}}

\def\GL{\textrm{GL}}

\def\SL{\textsf{SL}}
\def\SU{\textsf{SU}}

\DeclareMathOperator{\Pic}{Pic}

\DeclareMathOperator{\Bun}{Bun}
\DeclareMathOperator{\Gr}{Gr}
\DeclareMathOperator{\GR}{GR}

\DeclareMathOperator{\Hom}{Hom}
\DeclareMathOperator{\Isom}{Isom}
\DeclareMathOperator{\Spec}{Spec}

\DeclareMathOperator{\Aut}{Aut}

\DeclareMathOperator{\Sym}{Sym}
\DeclareMathOperator{\Lie}{Lie}

\DeclareMathOperator{\Gal}{Gal}
\DeclareMathOperator{\Ram}{Ram}
\DeclareMathOperator{\im}{Im}
\DeclareMathOperator{\sep}{sep}
\DeclareMathOperator{\Ad}{Ad}
\DeclareMathOperator{\Res}{Res}

\def\pr{\mathit{pr}}

\def\univ{\textrm{\tiny univ}}
\def\fin{\textrm{\tiny fin}}

\def\1halb{\frac{1}{2}}
\def\tto{\twoheadrightarrow}
\def\pprime{{\prime\prime}}

\def\sxymat{\xymatrix@C=1.5ex@R=0.8ex}
\def\grp{$\xymatrix{ R\times_{X}R  \ar[r]^-{\mu} & R \ar@<1ex>[r]^-{s}\ar@<-1ex>[r]_-{t} & X}$}
\def\dar{\ar@<-0.5ex>[r]\ar@<0.5ex>[r]}
\def\tar{\ar[r]\ar@<1ex>[r]\ar@<-1ex>[r]}
\newcommand{\dmap}[2]{\ar@<-0.5ex>[r]_-{#2}\ar@<0.5ex>[r]^-{#1}}
\newcommand{\dotarrow}[2]{\xymatrix{{#1}\ar@{..>}[r]&{#2}}}
\def\cart{\ar@{}[dr]|{\square}}

\def\cA{\mathcal{A}}
\def\cB{\mathcal{B}}

\def\cE{\mathcal{E}}

\def\cG{\mathcal{G}}
\def\cH{\mathcal{H}}
\def\cI{\mathcal{I}}

\def\cL{\ensuremath{\mathcal{L}}}

\def\cO{\mathcal{O}}
\def\cP{\mathcal{P}}
\def\cQ{\mathcal{Q}}

\def\cT{\mathcal{T}}
\def\cU{\mathcal{U}}

\def\cX{\mathcal{X}}

\def\cZ{\mathcal{Z}}

\def\cm{\mathfrak{m}}

\def\bA{{\mathbb A}}

\def\bG{{\mathbb G}}
\def\bP{{\mathbb P}}

\def\bN{{\mathbb N}}
\def\bZ{{\mathbb Z}}

\def\bR{{\mathbf R}}

\def\Ox{\widehat{\cO_x}}

\begin{document}
\title[Uniformization of $\cG$-bundles]{Uniformization of $\cG$-bundles}
\author{Jochen Heinloth}
\address{University of Amsterdam, Korteweg-de Vries Institute for Mathematics, Science Park 904, 1098 XH Amsterdam, The Netherlands}
\email{J.Heinloth@uva.nl}

\begin{abstract}
We show some of the conjectures of Pappas and Rapoport concerning the moduli stack $\Bun_\cG$ of $\cG$-torsors on a curve $C$, where $\cG$ is a semisimple Bruhat-Tits group scheme on $C$. In particular we prove the analog of the uniformization theorem of Drinfeld-Simpson in this setting. Furthermore we apply this to compute the connected components of these moduli stacks and to calculate the Picard group of $\Bun_\cG$ in case $\cG$ is simply connected.
\end{abstract}
\maketitle

The uniformization of the moduli stack of principal bundles on a smooth projective curve $C$ by the affine Gra\ss mannian proved by Drinfeld and Simpson \cite{Drinfeld-Simpson} has been proven to be a very useful tool (\cite{BeauvilleLaszlo},\cite{LaszloSorger},\cite{BeilinsonDrinfeld}). More recently the moduli spaces of torsors under non-constant group schemes over a smooth projective curve have been considered, in particular in the case of unitary groups. Motivated by their work on twisted flag manifolds \cite{Pappas-Rapoport} Pappas and Rapoport conjectured \cite{Pappas-Rapoport-Buendel} that these moduli spaces should have a similar uniformization by twisted affine flag varieties. Furthermore, they made conjectures on the geometry of the moduli stack of torsors under such group schemes which generalize those results on the moduli of principal bundles which have been proven using the affine Gra\ss mannian.

As a first step towards these conjectures we want to explain a generalization of the approach of Drinfeld and Simpson \cite{Drinfeld-Simpson} to this situation.

To state our results we need to introduce some notation:
We fix a smooth projective curve $C$ over a field $k$. Let $\cG$ be a smooth affine group scheme over $C$ satisfying the following conditions:
\begin{enumerate}
\item All geometric fibres of $\cG$ are connected.
\item The generic fiber of $\cG$ is semisimple.
\item Let $\Ram(\cG)\subset C$ be the finite set of points $x\in C$ such that the fiber $\cG_x$ is not semisimple and denote by $\Ox$ the complete local ring at $x$. Then $\cG_{\Ox}$ is a parahoric group scheme over $\Spec \Ox$ as defined by Bruhat-Tits (\cite{BT2}, D\'efinition 5.2.6).
\end{enumerate}
We will call such a group scheme a {\em (parahoric) Bruhat-Tits group scheme} over $C$. We will denote by $\Bun_\cG$ the moduli stack of $\cG$-torsors on $C$.

To motivate the study of these group schemes, let us recall the basic examples, which also provide alternative interpretations of some better known moduli spaces:
\begin{example*}
\begin{enumerate}
 \item The standard examples are constant group schemes: For a given semisimple group $G$ over $k$ we may consider the group scheme  $\cG:=G\times C$. In this case $\Bun_\cG$ is just the space of $G$-bundles on $C$. Similarly any $G$-torsor $\cP$ on $C$ defines a group scheme $\Aut_G(\cP/C)$ over $C$, which again is a Bruhat-Tits group scheme. 
 \item Moduli spaces of parabolic bundles are also of the form $\Bun_\cG$:  Again one starts with a semisimple group $G$ over $k$, together with a finite set of points $x_i\in C$ and choices of parabolic subgroups $P_i\subset G$. In this situation Bruhat-Tits construct a group scheme $\cG$ over $C$ together with a map $\cG \to G\times C$ such that the image of $\cG_{x_i}$ in $G=G\times x_i$ is the subgroup $P_i$ and the $\cO_{C,x_i}$-valued points $\cG(\cO_{C,x_i})$ are given by the $\cO_{C,x_i}$-valued points of $G\times C$ which reduce to elements of $P_i$ modulo the maximal ideal of $\cO_{C,x_i}$ i.e., the corresponding parahoric subgroup in $G(\cO_{C,x_i})$.
(The construction of Bruhat-Tits is most often phrased over a complete valuation ring. However, for the present example \cite{BT2} Th\'eor\`eme 3.8.1 can be applied, if one notices that the extensions of the root groups $U_a$ are defined over $C$.)
 \item More interesting examples are obtained by taking Weil restrictions and invariants: If $\pi:\tilde{C}\to C$ is a generically \'etale covering of $C$ and $\cG_0$ is a group scheme over $\tilde{C}$ of the types described above then the Weil restriction $\text{Res}_{\tilde{C}/C}\cG_0$ -- i.e., the group scheme whose sheaf of sections is given by the sheaf push-forward $\pi_*\cG_0$ -- is again a parahoric group scheme\footnote{Let us sketch how to see this: Pick a split maximal torus $\cT_0\subset \cG_0$. The main point is to note that adjunction gives a canonical isomorphism of cocharacter groups $X_*(\pi_* \cT_0)\cong X_*(\cT_0)$. Also the question is local, so we can restrict to $\Spec(\widehat{\cO}_{C,x})$. There, one can check explicitly that the Weil restriction of the root groups $U_a$ are the extensions needed for $\pi_*(\cG_0)$.}. Such group schemes are called induced group schemes.

Moreover, if $\pi:\tilde{C} \to C$ is a tamely ramified Galois-covering with group $\Gamma$, acting on a semisimple simply connected group $G$ then one can take invariants $\text{Res}_{\tilde{C}/C} (G\times \tilde{C})^{\Gamma}$ to obtain another example of Bruhat-Tits group schemes. In particular if $\Gamma=\bZ/2\bZ$ and $G=\SL_n$ then one can obtain the quasi-split unitary group $\SU_{\tilde{C}/C}(n)$ in this way. Here $\Bun_\cG$ classifies vector bundles on $\tilde{C}$ equipped with a hermitian form.
\item The construction of the previous point is interesting, even for tori (these groups are not semisimple). If one starts with the trivial torus $\bG_m$ on $\tilde{C}$ then finds groups $\cT$ the torsors under which are parametrized by Prym-varieties. 
\end{enumerate}
\end{example*}

Our main theorem is the confirmation of the uniformization conjecture of Pappas and Rapoport, which holds over every ground field $k$:
\begin{theorem}\label{Uniformisierung}
Fix a closed point $x\in C$. Let $S$ be a noetherian scheme and $\cP\in \Bun_\cG(S)$ a $\cG$-torsor on $C\times S$. Then there exists a faithfully flat covering $S^\prime \to S$ such that $\cP|_{(C-x)\times S^\prime}$ is trivial.
\end{theorem}

Some of the conjectures concerning the geometry of $\Bun_\cG$ can be deduced from the above result. To formulate these let us denote the generic point of $C$ by $\eta:=\Spec(k(C))$ and by $\overline{\eta}:=\Spec(k(C)^{\sep})$ the point of a separable closure of $k(C)$.
\begin{theorem}\label{Zusammenhangskomponenten}
If $\cG_\eta$ is simply connected, then $\Bun_\cG$ is connected. 

For general $\cG$ we have:
$$\pi_0(\Bun_\cG) \cong \pi_1(\cG_{\overline{\eta}})_{\Gal(k(C)^{\sep}/k(C))}.$$
\end{theorem}
For curves over finite fields, Behrend and Dhillon showed (\cite{BehrendDhillon} Theorem 3.3 and 3.5) that under some additional hypothesis on $\cG$ the above theorem would follow, if one could prove that the Tamagawa number of $\cG$ equals the number of elements of $\pi_1(\cG_{\overline{\eta}})_{\text{Gal}({k(C)}^{\textrm{sep}}/k(C))}$. 

\begin{theorem}\label{PicBunG}
Assume that $k$ is an algebraically closed field and that $\cG_{k(C)}$ is semisimple, absolutely simple and splits over a tamely ramified extension. For any $x\in \Ram(\cG)$ denote by $X^*(\cG_x):=\Hom(\cG_x,\bG_m)$ the character group of the fibre of $\cG$ over $x$. Then there is an exact sequence:
$$ 0 \to \prod_{x\in \Ram(\cG)} X^*(\cG_x) \to \Pic(\Bun_\cG) \to \bZ \to 0,$$
where the right arrow can be computed as a multiple of the central charge homomorphism at any given point $x\in C$.
\end{theorem}
Here the central charge homomorphism at a point $x\in C$ is given by pulling back a line bundle on $\Bun_\cG$ to the affine flag variety parameterizing bundles, together with a trivialization on $C-x$ and then applying the central charge homomorphism on the flag variety as constructed by Pappas and Rapoport \cite{Pappas-Rapoport} Section 10. This is recalled in more detail in Section 5 below.

Finally, again assuming that $\cG_{k(C)}$ splits over a tamely ramified extension of $k(C)$, we also prove an analog of the existence of reductions to generic Borel subgroups (\cite{Drinfeld-Simpson} Theorem 1) in the case of $\cG$-torsors, see Corollary \ref{Reduktion_auf_Borel}.

We would like to stress that these theorems are well known in the case of constant group schemes and a considerable part of the proofs of our results follow the lines of the proofs in this special case. The main difference of our approach is that we avoid the reduction to Borel subgroups in our proof. Instead of this we use a variant of the ``key observation'' in \cite{Pappas-Rapoport} to show that every tangent vector to the (twisted) affine flag manifold lies in the image of a map of the affine line into the affine flag manifold. The applications to the geometry of $\Bun_\cG$ are variations of the arguments in the case of constant group schemes, as explained in Faltings' article \cite{Faltings_Loopgroups} - of course some technical problems arise here. For example we find that the central charge morphism may not be independent of the point $x\in C$, but it may drop at points in $\Ram(\cG)$.

\noindent {\it Note added on revised version:} We will see that the proof of the uniformization theorem (Theorem \ref{Uniformisierung}) does not use the assumption that for $x\in \Ram(\cG)$ the group $\cG|_{\cO_x}$ is parahoric, so general connected Bruhat-Tits groups would do here. In the other theorems, we do however make use of the additional assumption.

\noindent {\bf Acknowledgements:} I thank M.\ Rapoport for explaining the conjectures formulated in \cite{Pappas-Rapoport-Buendel} to me and G.\ Harder for his explanations on Bruhat-Tits groups. I thank N.\ Naumann, a discussion with him on a related question in an arithmetic situation was the starting point for this article. Furthermore I am indebted to Y.\ Laszlo. He suggested many improvements on a previous approach to the main theorem of this article, in particular he suggested an argument helping to avoid the use of the strong approximation theorem and reduction to positive characteristics. I thank A.\ Schmitt his comments. I am indebted to the referee for many comments and corrections.

\noindent {\bf Notations:} $S$ will denote a noetherian base scheme defined over a field or an excellent Dedekind domain (for all our purposes it will be sufficient to assume that this is either the spectrum of a field or a smooth curve over a field).

$C\to S$ is smooth, projective, absolutely irreducible curve over $S$, i.e., the morphism $C\to S$ is smooth, projective of relative dimension $1$, such that all geometric fibres are connected.

For any $S$-scheme $T$ we will denote the base change from $S$ to $T$ by a lower index $T$, e.g., $C_T=C\times_S T$.
If $T=\Spec(R)$ happens to be affine then we will denote the base change to $\Spec(R)$ by a lower index $R$, e.g., $C_R:=C\times_S \Spec(R)$.

For a smooth, affine, group scheme $\cG$ over $C$ we will denote by $\Lie(\cG)$ its Lie-algebra, which is a vector bundle over $C$.

Given a $\cG$-torsor $\cP$ on $C$ and a scheme $F\to C$, affine over $C$ on which $\cG$ acts, we will denote by $\cP\times^\cG F$ the associated fibre bundle over $C$, i.e., $\cP\times^\cG F:= \cP \times_C F/\cG$ where $\cG$ acts diagonally on $\cP \times_C F$. Since $\cP$ is locally trivial for the \'etale topology descent for affine schemes implies that the quotient $\cP \times_C F/\cG$ is a scheme.

Also we will write $\cP_\univ$ for the universal $\cG$-torsor on $\Bun_\cG \times C$.


\section{Preliminaries on moduli stacks of torsors}\label{prel}

In this preliminary section we recall the basic results concerning the moduli stacks $\Bun_\cG$ that we will frequently use:
\begin{proposition}\label{BunG}
Let $\cG$ be a smooth affine group scheme over $C$, which is separated and of finite type. Denote by $\Bun_\cG$ the stack of $\cG$-torsors on $C$. Then $\Bun_\cG$ is a smooth algebraic stack, which is locally of finite type.
\end{proposition}

\begin{proof} This is certainly well-known, but I couldn't find a reference, so let us briefly indicate why Artin's criteria (\cite{LMB} Corollaire 10.11 and Remarque 10.12, where one also finds the necessary references on deformation theory) hold for $\Bun_\cG$:
First we need to check that given a $S$-scheme $T$, and $\cP,\cQ\in \Bun_\cG(T)$ the sheaf $\Isom(\cP,\cQ)$ is representable. Considered as sheaf over $C\times T$ the $\cG$-isomorphisms of $\cP$ and $\cQ$ are given by $\cP\times^\cG \cQ =(\cP\times \cQ)/\cG$ (where $\cG$ acts diagonally on $\cP\times \cQ$), which is affine over $C\times T$. Thus the sheaf $\Isom(\cP,\cQ)$ is the sheaf of sections over $C$ of $(\cP\times^\cG \cQ)$. This sheaf is representable by an separated scheme of finite type, because $C$ is projective. Thus we have shown that the diagonal of $\Bun_\cG$ is representable, separated and of finite type.

To verify the other criteria, we need to recall deformation theory of $\cG$-torsors: Let $A$ be a local Artin ring with $\Spec(A)\to S$. The maximal ideal of $A$ is denoted $\cm$ and denote the residue field $A/\cm=k$. Let $I\subset A$ an ideal with $I\cm=0$. Let $\overline{\cP}\in \Bun_\cG(A/I)$ and denote the special fibre of this family by $\cP_0:= \overline{\cP}_{k}$. 
Finally write $\Ad(\overline{\cP}):=\overline{\cP}\times^\cG \Lie(\cG)$ for the vector bundle defined by $\overline{\cP}$ via the adjoint representation of $\cG$.

We need to study the possible extensions of $\overline{\cP}$ to a family $\cP\in \Bun_\cG(A)$.
There exists an \'etale covering $U\tto C$ with $U$ affine, such that $\cP_0|_{U_k}=\cG \times U_{k}$ is trivial. Applying the lifting criterion for smoothness to $\overline{\cP}\to C_{A/I}$, we find that this trivialization can be lifted to a trivialization of $\overline{\cP}|_{U_{A/I}}$.

Thus $\overline{\cP}$ is given by a \v{C}ech-cocycle $\overline{g}_U\in \cG((U\times_C U)_{A/I})$. Again, $U$ being affine and $\cG$ being smooth, this $\overline{g}$ can be lifted to an element $g\in \cG((U\times_C U)_{A})$, and the possible such $g$ form a torsor under $H^0(U\times_C U)_{A/I},\Ad(\overline{\cP})\tensor_{A/I} I)$. The obstruction to modify $g$ to satisfy the cocycle condition defines an element in $H^2(C_{A/I},\Ad(\overline{\cP})\tensor_{A/I} I)=0$. So we see that we can aways find an extension $\cP\in \Bun_\cG(A)$. 
Moreover, we see that the possible extensions $\cP$ of $\overline{\cP}$ are parameterized by $H^1(C_{A/I},\Ad(\overline{\cP})\tensor_{A/I} I)$, and the automorphisms of such extensions (i.e., automorphisms of $\cP$, inducing the identity on $\overline{\cP}$) are parameterized by $H^0(C_{A/I},\Ad(\overline{\cP})\tensor_{A/I} I)$. Since the groups $H^i(C,\Ad(\cP_0))$ are finite dimensional vector spaces, this implies that for every $\cG$-torsor $\cP_0\in\Bun_\cG(k)$ there exists a versal deformation. Any such formal deformation is algebraizable by Grothendieck's existence theorem (EGA III, Section 5).

Finally we need to check that versality is an open condition. In the above discussion we already noted that for every $\cP_0\in \Bun_{\cG}(k)$ the fibre of tangent stack (which by definition is the algebraic stack given on affine schemes by $\Spec(R) \mapsto \Bun_{\cG}(R[\epsilon]/(\epsilon^2))$, see \cite{LMB} D\'efinition 17.13) is isomorphic to the stack-quotient of the vector spaces $[H^1(C_k,\Ad(\cP_0))/H^0(C_k,\Ad(\cP_0))]$. Since $C$ is a curve, the formation of $\bR^1 p_* \Ad(\cP)$ commutes with base-change. In particular for a family $\Spec(R) \to \Bun_\cG$ to induce a surjection on tangent spaces is equivalent to the condition that the induced Kodaira-Spencer-map from the tangent sheaf of $\Spec(R)$ to the pull back of $\bR^1 p_{\Bun_{\cG},*} \Ad(\cP_{\univ})$ to $\Spec(R)$ is surjective and this is an open condition.
\end{proof}

\section{Preliminaries on twisted affine flag manifolds}

We will need to use the construction of loop groups and twisted flag manifolds in families. This is well known for constant groups see e.g., the Appendix of \cite{Gaitsgory_CentralElements}, so we just have to check that a similar construction works in our situation, in particular at those points of $C$ where the group scheme is not semisimple. Similarly we have to give the analog of the moduli interpretation of the affine Gra\ss mannians, parameterizing torsors trivialized outside a point of $C$.

In this section we will assume that our family $C\to S$ has a section $s\colon S\to C$ and $\pi\colon \cG\to C$ will be a smooth affine group scheme over $C$.

In order to reduce to the case of $\GL_n$ we have to make some technical assumptions, these will automatically be satisfied in the situation considered in the introduction.
\begin{assumption}\label{Annahme}
    \begin{enumerate}
    \item The conditions of \cite{Pappas-Rapoport} Proposition 1.3 are satisfied for our family i.e.: There exists a vector bundle $\cE$ on $C$ and a faithful representation $\rho\colon \cG \to \GL(\cE) \times \bG_m$ such that $\rho$ is a closed embedding and the quotient $\cQ=(\GL(\cE)\times \bG_m)/\cG\to C$ is representable and quasi-affine.
    \item Denote by $\cO_C[\cG]$ the $\cO_C$-algebra of functions on $\cG$. Assume that locally in the Zariski topology on $C$ there exist vector bundles $\cE,\cI$ on $C$ which are $\cO_C[\cG]$-comodules (i.e., $\cG$ acts on $\cE,\cI$) and a $\cG$-equivariant exact sequence:
    $$ \cI \tensor \Sym^\bullet \cE \to \Sym^\bullet\cE \to \cO_C[\cG]\to 0.$$ 
    To relate this to condition (1), note that given a vector bundle $\cE_0$ on $C$ and a faithful representation $\rho\colon \cG \to \SL(\cE_0)$ then we get a closed embedding $i\colon \cG\to \SL(\cE_0)\to \cE\text{nd}(\cE_0)$ and denote the ideal sheaf of $i(\cG)$ by $\cI_0$. Then we can can use  $\cE:=\cE\text{nd}(\cE_0)$ and then choose a $\cG$-equivariant vector bundle $\phi\colon  \cI\to \cI_0$ such that the image of $\phi$ generates $\cI_0$.
    \end{enumerate}
\end{assumption}

\begin{example*}\label{Beispieldarstellung}
We are interested in the following situations:
    \begin{enumerate}
    \item $C/k$ and $\cG$ are as in the introduction and $s\in C(k)$ is a rational point. This is the situation considered in \cite{Pappas-Rapoport} Section 1.b. Since $C$ is regular of dimension $1$ the group scheme $\cG$ always admits a faithful representation as above, at least locally as is shown in loc.cit. by considering a subrepresentation of the regular representation $\pi_*(\cO_\cG)$ of $\cG$. Given an open subset $U\subset C$ any representation on a vector bundle $\cE_U\incl{} \pi_*(\cO_\cG)|_U$ defined over $U$ extends to the flat closure $\cE \subset \pi_*(\cO_\cG)$. Thus, taking a direct sum of such representations we also find that $\cG$ has a faithful representation over $C$. The assumption on the quotient $\cQ$ is unchanged if we add representations, because given two vector bundles $\cE_1,\cE_2$ we have  $\GL(E_1\oplus E_2)/\cG= \GL(\cE_1\oplus \cE_2) \times^{\GL(\cE_1)\times\GL(\cE_2)} \GL(\cE_1)\times \GL(\cE_2)/\cG$ and the quotient $\GL_{n_1+n_2}/\GL_{n_1}\times \GL_{n_2}$ is affine.
    \item Given $C/k$ and $\cG/C$ as in (1) we can vary the point in $C$ as follows. Consider the constant family $\pr_2\colon C\times C \to S:=C$ given by the projection on the second factor together with the diagonal section $\Delta\colon  S=C \to C\times C$ and $\cG_C:=\cG \times C\to C\times C$. The representation from (1) can be pulled back to $C\times C$, so again the extra condition is automatically satisfied.
    \item Given $C/k$ and $\cG$ as in (1) we can extend everything to a family over a finitely generated $\bZ$-algebra: The schemes $C,\cG$, the section $s$, the faithful representation and the quasi-affine quotient are defined over some finitely generated $\bZ$-algebra $A$ and we can take $S=\Spec(A)$. The same holds for the second assumption.
    \end{enumerate}
\end{example*}

Let us fix a notation for completions. Given an affine scheme $X=\Spec(A)$ and a closed subscheme $Z=\Spec(A/I)\subset X$, we denote by $\widehat{X}_Z:=\Spec(\widehat{A_I})$ the spectrum of the completion of $A$ along $I$. We use the same notation for non-affine $X$ if $Z$ is contained in an affine subset of $X$.
If $Z$ is given by the image of a closed embedding $s\colon Z \to X$ we write $\widehat{X}_s:=\widehat{X}_{s(Z)}$. In particular in our situation  $\widehat{C_s}$ is the completion of $C$ along $s$ and we denote by $\mathring{C}:=C-s$ the complement of the image of the section $s$. We define the following functors on affine schemes $T=\Spec(R)$ over $S$:
\begin{enumerate}
\item $L^+\cG(T):=\cG(\widehat{(C\times T)}_{s\times T})$.
\item $L\cG(T):=\cG(\widehat{(C\times T)}_{s\times T}\times_{C} (\mathring{C}))$
\item $\GR_{\cG}(T):=(L\cG/L^+\cG)^\#(T)$, where the sheaffification $^\#$ is taken in the fpqc-topology.
\end{enumerate}
In particular if $S=\Spec(k)$ and $\cG=G\times C$ is a constant group scheme then these functors coincide with the classical loop groups and the affine Gra\ss mannian.

\begin{proposition}\label{LG} Under Assumption \ref{Annahme} we have:
\begin{enumerate}
\item The functor $L^+\cG$ is representable by an $S$-scheme.
\item The functor $L\cG$ is representable by an ind-scheme over $S$.
\item The functor $\GR_\cG$ is representable by an ind-scheme over $S$ and the morphism $L\cG\to \GR_\cG$ admits local sections in the \'etale topology.
\end{enumerate}
\end{proposition}
\begin{proof}
The questions are local in $S$, so we may assume that $S=\Spec(A)$ is affine, that the vector bundle $\cE$ from Assumption \ref{Annahme} (2)  is trivial on an open neighbourhood $U\subset C$ of $s$ and that there exist a function $t$ on $U$, which is a local parameter at $s$ i.e., such that $s$ is the subscheme defined by $t$.

In this case we may argue as in the case of constant group schemes: The representation $\rho$ defines a closed embedding $\cG|_U\subset \bA_S^{(n+1)^2} \times_S U$.  We have defined $L^+\bA^n(R)=R[[t]]^{(n+1)^2}$ so this functor is represented by an infinite affine space $\prod_{i\in \bN} \bA_S^{(n+1)^2}$ and the subfunctor $L^+\cG(R)=\cG(R[[t]])$ is given by an infinite set of polynomial equations, so it is again representable.

The same argument holds for the ind-scheme $L\cG$: In the case of $\cG=\GL_n$ the functor $L\cG$ is the union of the subfunctors given by those matrices for which the entries are Laurent series of the form $\sum_{i>-N} a_i t^i$, which are representable for every $N$. In the general case we use a representation $\cG\subset \GL_n$ to describe the functor $L\cG$ as a subfunctor of $L\GL_n$, given by polynomial equations. In this way we obtain ind-schemes, defined over open subsets of $S$. Note that the schemes occurring in the inductive limit glue, because we can compare the different choices of local parameters on the intersections. This is because a different choice of a local parameter $t^\prime\in \cO_U$ can be expanded as a power series in $\widehat{\cO}_{U,s}=A[[t]]$ as 
$t^\prime = \alpha t
+\sum_{i>1} \alpha_i t^i$ with $\alpha\in A^*$. In particular ${t^\prime}^{-1}=t^{-1}(\alpha + \sum_{i>0} \alpha_{i+1} t^i)^{-1}$ and the second factor is an invertible powerseries, so we see that the different choices of $t$ respect the subfunctors defined above using the pole order. So we can glue the subschemes to obtain an
ind-scheme over $S$.

Finally the argument for the affine Gra\ss mannian given in \cite{Pappas-Rapoport} Theorem 1.4 generalizes as well: By Assumption \ref{Annahme} we have an embedding $\cG\subset \GL_n(\cE)\times \bG_m=:\cH$ such that the quotient $\cQ=\cH/\cG$ is quasi-affine over $C$, i.e. $\cQ$ is an open subscheme of a scheme $\cZ$, which is affine over $C$.  Since $\cG$ is smooth the map $p\colon\cH\to \cQ$ is smooth, so we can lift any section of $\cQ$ locally in the \'etale topology to a section of $\cH$. Since for any henselian local $\Spec(R)\to S$ we can find a local coordinate for the section $s_R$ we obtain $L^+\cQ(R)=\cQ(R[[t]])$ and $R[[t]]$ is again henselian. Therefore we find $L^+\cQ=L^+\cH/L^+\cG$ as fpqc-sheaves. 

We write $Lq\colon L\cH \to L\cQ$ for the map induced by $q\colon\cH\to \cQ$ and $e_\cQ\colon S\to L^+\cQ \subset L\cQ$ for the section induced by the neutral element of $L\cH$. Note that $L\cG=(Lq)^{-1}(e_\cQ)$. 

By construction $L^+\cZ \hookrightarrow L\cZ$ is a closed subscheme, so that $L^+\cQ\subset L\cZ$ is locally closed. Therefore the preimage $(Lq)^{-1}(L^+\cQ)\subset L\cH$ is a locally closed sub-ind-scheme, which is $L^+\cH$ invariant.
Also any element $h\in (Lq)^{-1}(L^+\cQ)(R)$ defines $Lq(h)\in L^+\cQ(R)$, so \'etale locally on $R$ we can write $h=g\cdot h^+$ with $g\in L\cG$ and $h\in L^+\cH$. 

Since the $L^+\cH$-torsor $L\cH \to \GR_\cH$ is Zariski-locally trivial, this implies that the image $Y$ of $Lq^{-1}(L^+\cQ)$ in $\GR_\cH$ is again locally closed. We already noted, that $L\cG$ maps to $Lq^{-1}(L^+\cQ)$, so that we get a map $\pi\colon L\cG \to Y$ and we claim that this identifies $L\cG/L^+\cG\cong Y$: 

First, since $L^+\cG \subset L^+\cH$, the map $\pi$ factors through $\GR_\cG$. Let us show that the map has sections locally in the \'etale topology. The map $L\cH\to \GR_\cH$ being a Zariski locally trivial $L^+\cH$-torsor, we can lift any $R$-valued point of $Y$ Zariski locally to a point of $Lq^{-1}(L^+\cQ)$ and locally in the \'etale topology we have just seen that up to multiplication by an element of $L^+\cH$ such a point can be lifted to $L\cG$.

Finally, any two elements $h,h^\prime$ in a fibre $L\cG\to Y$ differ by an element of $L\cG\cap L^+(\cH)=L^+\cG$, so the map $L\cG\to Y$ is indeed a $L^+\cG$-torsor.
\end{proof}

\begin{notation}
If $C$ is a curve over $S=Spec(k)$ any point $x\in C(k)$ can be viewed as a section of $C\to \Spec(k)$. In this case we will denote by $\Gr_{\cG,x}$ the ind-scheme over $S=\Spec(k)$
constructed above. Note that in the setting of example (2) this can also be viewed as the
fiber of $\GR_\cG$ over $x$. Similarly $L\cG_x$ will denote the fiber of $L\cG$ over $x$.
This is a slight abuse of notation, we really consider $(L\cG)_x$, which is the loop group as defined in \cite{Pappas-Rapoport} and not the standard
loop group of the fiber $\cG_x$, which would be a loop group of a possibly non-semisimple group.  This should not cause confusion, since the letter $\cG$
indicates that we are working with non-constant group schemes.
\end{notation}
\begin{remark}
Note that if $\cG_x$ is not semisimple, then $\Gr_{\cG,x}$ is a (twisted) affine flag manifold and not an affine Gra\ss mannian. In particular the family $\GR$ constructed above is different from the family used by Gaitsgory \cite{Gaitsgory_CentralElements} to construct the center of the Iwahori-Hecke algebra. In particular the above family does not contain the extra $G/B$-factor in the fibres over those points where $\cG_x$ is semisimple.
\end{remark}
Next, we need to recall the construction and the basic properties of the map from the twisted affine flag manifold to the moduli stack of $\cG$ torsors. This is certainly well-known, however we could not find a reference for the case of non-constant group-schemes. 
\begin{proposition}\label{Uniformisierungs_Abbildung}
\begin{enumerate}
\item 
The ind-scheme $L\cG$ represents the functor given on noetherian rings $R$ over $S$ by
$\cG$-torsors $P$ on $C_R$ together with trivializations on $C_R-s_R$ and the formal completion
$\widehat{C_{s_R}}$.
\item The ind-scheme $\GR_\cG$ represents the functor given on noetherian rings $R$ over $S$ by
$\cG$-torsors together with a trivialization on $C_R-s_R$. In particular, there is a natural forgetful map $p\colon \GR_\cG \to \Bun_{\cG}$.
\item The map $p\colon \GR_\cG \to \Bun_\cG$ is formally smooth.
\item For any point $x\in C(k)$ the map $p_x\colon  \Gr_{\cG,x} \to \Bun_{\cG}$ is formally smooth. In particular, let $P\in \Gr_{\cG,x}$ be a point defining the $\cG$-torsor $p_x(P)=:\cP$ on $C$, then the map $p$ induces a surjection of tangent spaces
$$dp_P\colon  T_{\Gr_{\cG,x},P} \tto H^1(C,\cP\times^\cG \Lie(\cG)).$$
\end{enumerate}
\end{proposition}
\begin{proof}
First of all, given a $\cG$-torsor $P$ on $C_R$, together with trivializations on $C_R-s_R$ and $\widehat{C_{s_R}}$ the difference of the two trivializations on $\widehat{C_{s_R}}\times_{C_R} (C_R-s_R)$ is an element of $L\cG(R)$.
Furthermore if we change the trivialization on $\widehat{C_{s_R}}$ we obtain the corresponding element in $L\cG(R)$ by multiplication with an element of $L^+\cG(R)$.

Similarly, if only a trivialization on $C_R-s_R$ is given, we can \'etale locally on $R$ choose a trivialization of $P$ on $\widehat{C_{s_R}}$ and this defines an element in $\GR_\cG$ which is independent of the chosen trivialization.

To prove parts (1) and (2) of the proposition, we therefore have to give an inverse to this construction, i.e. we have to construct the map $p$.

Since $\GR_\cG$ is the inductive limit of noetherian schemes it is sufficient to
construct the map $p$ on $T$-valued points, where $T=\Spec(R)$ is a noetherian affine
$S$-scheme (for non-noetherian rings one may use the technique of Beauville-Laszlo
\cite{BeauvilleLaszlo}, but we will not use this). Furthermore, since $L\cG \to \GR_\cG$
admits local sections in the \'etale topology, it is sufficient to construct an
$L\cG^+$ equivariant map $\tilde{p}\colon L\cG(R) \to \Bun_\cG(R)$ for noetherian $S$-algebras
$R$.

Since $R$ is noetherian the morphism  $\widehat{(C_R)}_{s_R}\cup (C_R-s_R) \to C_R$
is a faithfully flat, quasi compact covering and $L\cG(R)=\cG((\widehat{(C_R)}_{s_R})\times_{C_R} (C_R-s_R))$.

To apply descent theory to this covering some care is needed\footnote{I would like to thank Y. Laszlo for
pointing out this problem}. Let us write $U:= \widehat{(C_R)}_{s_R}\cup (C_R-s_R)$. A descent datum for the trivial $\cG$-torsor on $U$ is an element of $\cG(U\times_C U)$ satisfying a cocycle condition in $\cG(U\times_C U \times_C U)$. One component of $U\times_C U$ is the selfintersection of the completion $\widehat{(C_R)}_{s_R}$ over $C$. The following lemma explains how any element of $L\cG(R)$ can be used to define a gluing cocycle on this scheme:

\begin{lemma}\label{push-out}
$R$ be a noetherian $S$-algebra. Let $C/S$ be smooth curve, $s\colon S\to C$ a section. Write
$D_R:=\widehat{(C_R)}_{s_R}$, $\mathring{C}:=C-s$ and $\mathring{D}_R:=D_R\times_{C_R}
(\mathring{C}_R)$.

Then $D_R\times_{C_R} D_R = D_R\cup_{\mathring{D}_R} D_R$ i.e., the following is a
pushout-diagram of schemes:
$$\xymatrix{
{\mathring{D}}_R \ar@{^(->}[r]\ar@{^(->}[d]^{\Delta} & D_R\ar@{^(->}[d] \\
\mathring{D}_R\times_{C_R} \mathring{D}_R \ar[r] & D_R\times_{C_R} D_R}.$$
\end{lemma}
\begin{proof}
First, the question whether $D_R\times_{C_R} D_R$ represents the push out of the above diagram is local in the Zariski topology on $S$. Moreover, the above fibred products only depend on a Zariski open neighbourhood of $s_R$ so we may further assume that $C=\Spec(A)$ is affine and that there exists $t\in A$ which is a local parameter at $s_R$.

Furthermore all schemes in the above diagram are affine, so the claim is equivalent to the dual statement that we have a
pull-back diagram of rings:
$$\xymatrix{
R((t)) & R[[t]]\ar[l] \\
R((t))\tensor_{A_R} R((t))\ar[u] & R[[t]]\tensor_{A_R} R[[t]].\ar[u]\ar[l]}$$ So we have
to show that given an element $\sum P_i(t)\tensor Q_i(t) \in R((t))\tensor_{A_R} R((t))$
such that $\sum P_i(t)Q_i(t) \in R[[t]]\subset R((t))$ we can find $p_i,q_i\in R[[t]]$
such that $\sum P_i(t)\tensor Q_i(t)=\sum p_i(t)\tensor p_i(t)$.

First note that $t\in A_R$ so we may assume that $Q_i=\sum_{n=0}^\infty  b_{i,n} t^n$
with $b_{i,0}=1\in R$. Write $P_i(t)=\sum_{n=-n_i}^\infty a_{i,n} t^n$ and do induction on
the maximal pole order $N$ of the $P_i$.  Write
\begin{align*}\sum_i P_i(t)\tensor
Q_i(t)&=\sum_{\{i|n_i=N\}} a_{i,N} t^{-N}\tensor 1 + \sum_{\{i|n_i=N\}} a_{i,N}t^{-N} \tensor
(Q_i-1)\\&+ \sum_{\{i|n_i=N\}} (P_i-a_{i,N}t^{-N}) \tensor Q_i + \sum_{\{i|n_i<N\}} P_i
\tensor Q_i.
\end{align*}

 Since $\sum P_i(t)Q_i(t) \in R[[t]]$ we see that $\sum_{i} a_{i,N}=0$, so the first term
 in the above sum vanishes and the other terms have poles of lower order. This proves our
 claim.
\end{proof}

Given $g\in \cG(\mathring{D}_R)=L\cG(R)$, the element $\pr_1^*(g)\times
\pr_2^*(g^{-1})\in \cG(\mathring{D}_R\times_{C_R} \mathring{D}_R)$ restricts to the identity on the diagonal $\mathring{D}_R\subset \mathring{D}_R\times_{C_R} \mathring{D}_R$, which clearly extends to $D_R$. By the above lemma $\pr_1^*(g)\times \pr_2^*(g^{-1})$ therefore defines an element $g_{1,1}\in \cG(D_R\times_C D_R)$. 

We claim that this is the gluing cocycle needed to apply descent as indicated before the Lemma. Namely $\cG(U\times_C U)=\cG(D_R\times_{C_R} \mathring{C}_R) \times \cG(D_R\times_{C_R} D_R) \times \cG(\mathring{C}_R\times_{C_R} D_R) \times \cG(\mathring{C}_R)$ and we have constructed the element $(g,g_{1,1},g^{-1},1)$ in this group. It is easy to verify that this element satisfies the cocycle condition on $U\times_{C_R} U\times_{C_R} U$, for example on $D_R\times_{C_R} \mathring{C}_R \times_{C_R} D_R\cong \mathring{D}_R \times_{C_R} \mathring{D}_R$ we have $\pr_{12}^*(g)\cdot \pr_{23}^*(g^{-1})=pr_{13}^*(g_{1,1})$ because $\pr_{13}$ is just the inclusion $\mathring{D}_R \times_{C_R} \mathring{D}_R\to {D}_R \times_{C_R} {D}_R$. This example already shows that the cocycle condition forces us to use the element $g_{1,1}\in \cG(D_R\times_{C_R} D_R)$.

In particular if $\cG = \GL_n \times C$, then $\Bun_\cG$ is the stack of vector bundles
on $C$ and thus fpqc-descent for vector bundles gives (1) and (2).

For general $\cG$ we use our Assumption \ref{Annahme} (2) in order to apply a Tannaka type argument:
Consider the exact sequence of $\cA$-comodules
\begin{equation}\label{sequenz} \cI \tensor \Sym^\bullet \cE \to \Sym^\bullet\cE \to \cA \to 0.
\end{equation}
In particular the representation $\rho$ defines a map $L\cG(R)\to L\GL_n(R)$ and thus we get gluing
cocycles for the vector bundles $\cE$ and $\cI$ on our faithfully flat covering.
Therefore this data defines vector bundles $\cE_R$, $\cI_R$ on $C_R$ together
with a map $ \cI_R \tensor \Sym^\bullet \cE_R \to \Sym^\bullet\cE_R$. The cokernel
$\cA^\prime_R$ of this morphism is an algebra on $C\times R$, which is locally isomorphic
to $\cA_R$ and the $\cA_R$-comodule structure also descends to $C_R$. Thus
$\Spec(\cA^\prime_R)$ is a $\cG$-torsor on $C_R$. Note that we only assumed that the
sequence \ref{sequenz} exists locally on $C$, so this construction defines a $\cG$
torsor on an open neighbourhood $U_s$ of the section $s$, together with a trivialization
on $U_s-s$. Since the trivial torsor extends canonically to $C-s$ this
is sufficient to construct $\tilde{p}$. This proves (1) and (2).

Part (3) and (4) of the proposition follow from (2) by the lifting criterion for formal smoothness: Again, since
$\Bun_\cG$ is locally noetherian, we may restrict to noetherian $S$-algebras $R$. Let
$I\subset R$ be a nilpotent ideal and let $\cP$ be a $\cG$-torsor on $C_R$. By (2), a
preimage of $\cP|C_{R/I}$ under $p$ is given by a section of
$\cP|\mathring{C}_{R/I}$. Since $\cP \to C_R$ is smooth and $\mathring{C}_R$ is affine we
can apply the lifting criterion to $\mathring{C}_{R/I}$ to obtain that any such section
can be extended to a section over $\mathring{C}_R$.
\end{proof}

The following is an application of the key observation 9.3 in \cite{Pappas-Rapoport}:
\begin{lemma}\label{Tangentialraum_hochheben}
Assume that $\cG_{k(C)}$ is simply connected. Let $x\in C$ be a closed point and denote by $K_x$ the corresponding local field. Finally let $g\in \Gr_{\cG,x}(\overline{k})$ be a geometric point.

For any finite dimensional subspace $V\subset T_{(\Gr_{\cG,x}),g}$ there exist a map $f\colon \bA_\kbar^n\to L\cG_x\to \Gr_{\cG,x}$ which induces a surjection $df\colon  \bA_\kbar^n \tto V$.
\end{lemma}
\begin{proof} 

Since by definition of $\Gr_x$ the group $L\cG_x$ acts transitively on $\Gr_x$ we may assume that $g=1\in \Gr_x$.

Let $v$ be an element of $T_{\Gr_x,1}$ i.e. $v\in \Gr_x(k[\epsilon]/\epsilon^2)$. Since the morphism $L\cG_x \to \Gr_x$ admits sections in the \'etale topology, we may choose a preimage $\tilde{v}\in L\cG(k[\epsilon]/\epsilon^2)$ with $\tilde{v}\equiv 1 \mod \epsilon$. In order to lift $\tilde{v}$ to a morphism of $\bA^1 \to L\cG_x$ we apply the same reductions as in \cite{Pappas-Rapoport} Section 8.e.2 and 9.a:

Since $\cG$ is simply connected, there exist finite extensions $K_i/K_x$ such that we can write $\cG_{K_x}=\prod \text{Res}_{K_i/K_x} G_i$ where $G_i$ are absolutely simple and simply connected groups over $K_i$. In particular to prove the Lemma we may assume that $\cG_{K_x}$ is absolutely simple.

We may assume that $k=\kbar$ is algebraically closed. Then the group $\cG_{K_x}$ is automatically quasi-split (\cite{Steinberg} p.78, last paragraph or \cite{Landvogt} Proposition 10.1). In this case, as in the construction of the Bruhat-Tits group scheme, we have an open embedding $U^- \times T \times  U \to \cG_{K_x}$, where $T$ is a maximal torus of $\cG_{K_x}$ and $U,U^-$ are products of root subgroups. Since $L\cG_x(k[\epsilon]/(\epsilon^2))=\cG_x(k[\epsilon]/(\epsilon^2)((t)))$ and $\tilde{v} \mod \epsilon\in U^-\times T\times U$ we know that $\tilde{v}\in U^-\times T\times U\subset \cG_{K_x}$.

Since $U$ and $U^-$ are affine spaces it is sufficient to show that every element $T(k[\epsilon]/(\epsilon^2)((t)))$ can be lifted to an element of $\cG_x(k[\epsilon]((t)))$.

Now $\cG_x$ is simply connected and therefore $T$ splits into a product of induced tori. Thus, as in \cite{Pappas-Rapoport} proof of 9.3 we may reduce ourselves to the case $\cG=\SL_2$ or $\cG=\SU_3$.  If $\cG=\SL_2$ the formula:
$$\mat c 0 0 {c^{-1}} = \mat 1 c 0 1 \mat 1 0 {-c^{-1}} 1 \mat 1 c 0 1 \mat 0 {-1} 1 0 $$
shows that any element of $T(k[\epsilon]/(\epsilon^2)((t)))$ can be lifted to an element of $\SL_2(k[\epsilon]((t)))$.
Similarly if $\cG=\SU_3$ is the unitary group for a quadratic extension $K^\prime_x/K_x$ formula (9.13) in \cite{Pappas-Rapoport} (we denote the generator of $\Gal(K^\prime_x/K_x)$ by $\overline{\,\cdot\,}$):
$$\dmat {-d} 0 0 0 {\frac{\overline{d}}{d}} 0 0 0 {-\frac{1}{\overline{d}}} = 
\dmat 1 {-c} {-d} 0 1 {\overline{c}} 0 0 1
\dmat 1 0 0 {\frac{\overline{c}}{\overline{d}}} 1 0 {-\frac{1}{\overline{d}}} {-\frac{c}{d}} 1
\dmat 1 {-\frac{c\overline{d}}{d}} {-d} 0 1 {\frac{\overline{c}d}{\overline{d}}} 001
\dmat 0 0 1 0 {-1} 0 1 0 0  $$
shows the claim in this case.
\end{proof}
\begin{remark}\label{Affin_ueber_Basis}
Note, that if in the above construction the groups $\cG,T,U$ and the decomposition of $T$
into a product of induced tori are  defined over a ring $R$ instead of a field $k$, then
the map $\bA^n\to L\cG_x$ will also be defined over $R$.
\end{remark}

\begin{corollary}\label{Glatte_Umgebungen}
Assume that $\cG_{k(C)}$ is simply connected. Let $\cP\in \Bun_{\cG}(\overline{k})$ be a $\cG$-torsor which lies in the image of the
map $\Gr_{\cG,x}\to \Bun_\cG$. Then there  exists a smooth neighbourhood $p\colon U_\cP\to
\Bun_\cG$ of $\cP$ such that $U_\cP\subset \bA^n$ and the map $p$ can be lifted to
$\tilde{p}\colon U_\cP \to \Gr_{\cG,x} \to \Bun_\cG$.
\end{corollary}

\begin{proof}
In the proof of Proposition \ref{BunG} we have seen that the tangent stack to $\Bun_{\cG}$ at $\cP$ is given by $[H^1(C,\cP\times^\cG
\Lie(\cG))/H^0(C,\cP\times^\cG \Lie(\cG))]$. By our assumption there exists a preimage
$(\cP,\phi)\in \Gr_{\cG,x}(\kbar)$ of $\cP$. Since the map $\Gr_{\cG,x} \to \Bun_{\cG}$
is formally smooth, there exist a finite dimensional subspace $V\subset
T_{\Gr_{\cG,x},(\cP,\phi)}$ such that $dp_x\colon V \to H^1(C,\cP\times^\cG \Lie(\cG))$ is
surjective. The above lemma shows, that there is a map $f\colon \bA^n \to \Gr_{\cG,x}$ such
that $f(0)=(\cP,\phi)$ and $df\colon \bA^n \to H^1(C,\cP\times^\cG \Lie(\cG))$ is surjective.
In particular the map $p\circ f\colon \bA^n \to \Bun_\cG$ is smooth at $0$. Thus there is a
Zariski open neighbourhood $U_{\cP}\subset \bA^n$ such that $p\circ f|U$ is smooth.
\end{proof}

\begin{corollary}\label{Bild offen}
Assume that $\cG_{k(C)}$ is simply connected. Then the image of the map $\pr_x\colon \Gr_x \to \Bun_\cG$ is open.
\end{corollary}
\begin{proof}
By the preceding corollary we know that for each point in the image of $\pr_x$ there is a
smooth neighbourhood contained in the image.
\end{proof}

\section{Reminder on local triviality}\label{local_triviality}
As a first step in our proof of the uniformization theorem, we need to recall a theorem of Steinberg and Borel--Springer. They showed that $\cG$-bundles over a curve
$C$, defined over an algebraically closed field are locally trivial in the Zariski topology. A theorem of Harder says that the same holds over finite fields if $\cG$ is simply connected. Together with the deformation arguments from the previous section this will suffice to deduce the uniformization theorem. 

Let us recall:
\begin{theorem*}[Steinberg, Borel--Springer\cite{BorelSpringer}]
Let $k(C)$ be the function field of a curve $C$, defined over an algebraically closed field $k$ and let $G$ be a connected semisimple group scheme over $k(C)$. Then
$$H^1(k(C),G)=0.$$
\end{theorem*}

This theorem tells us that if $\cP$ is a $\cG$ torsor on $C$, then the restriction of $\cP$ to the generic point of $C$ is trivial. Since $\cP$ is of finite type over $C$ any trivialization over the generic point will extend to a non-empty open subset $U\subset C$ i.e., we can find $U\subset C$ such that  $\cP|U$ is trivial. Finally, since  $\cG$ is smooth any trivialization over a closed point can be lifted to a trivialization over the formal completion, so for any point $x\in C(k)$ any $\cP$ torsor is trivial on the formal completion $\widehat{C_x}$.

Therefore any $\cG$-torsor $\cP$ can be obtained by gluing the trivial torsor on some open subset $U\subset C$ and the trivial torsors on the formal completions at the remaining points $x\in C-U$. In other words, denoting by $\bA_{k(C)}$ the ad\`eles of $k(C)$ we see that the map
$p_{\bA}\colon \cG(\bA_{k(C)})\to \Bun_\cG(k)$ which maps an ad\`ele $g$ to the torsor obtained from the cocycle given by $g$ is essentially surjective (recall that $\Bun_{\cG}(k)$ is a
category).

Recall furthermore that the above theorem also implies, that $\cG_{k(C)}$ is always quasi-split \cite{SerreCohomologieGaloisienne} III 2.2. (The proof of (i')$\Rightarrow$ (ii') in this reference does not use that the ground field is perfect.) 

\begin{remark} 
To prove the uniformization theorem we may always pass to an extension of the ground field, so we might always assume that $k$ is algebraically closed. However the deformation arguments of the previous section also allow to perform a reduction to the case of finite fields. In this case Harder proved in \cite{Harder_GaloiskohomologieIII} that $H^1(k(C),\cG)=0$ if $\cG$ is simply connected. It would therefore also suffice to use Harder's theorem in the following. 
\end{remark}

\section{The case of simply connected group schemes}

\begin{theorem}\label{EinfachzusammenhaengendeUniformisierung}
Let $k$ be a field, $\cG$ a simply connected parahoric Bruhat-Tits group scheme over $C$ and let $x\in C$ be a closed point.

Then, for every noetherian scheme $S$ and every family $\cP \in \Bun_\cG(S)$ there exists an \'etale covering $S^\prime \to S$ such that $\cP|(C-x)\times S^\prime$ is trivial.
\end{theorem}
\begin{proof}
By Proposition \ref{Uniformisierungs_Abbildung} and Corollary \ref{Bild offen} we know
that the map $\pr_x\colon \Gr_x\to \Bun_\cG$ is formally smooth with open image. Therefore we want to show that
this map is also essentially surjective on $\overline{k}$-valued points.

We claim that for all points $y\in C$ the image of $\pr_y$ coincides with the image of
$\pr_x$. First, since $\Gr_{\cG_x}$ is connected (\cite{Pappas-Rapoport} Theorem 0.1) we know that both images lie in the connected component of the trivial bundle in $\Bun_\cG$. Now let $\cP$ be a bundle which lies in the image of $\pr_y$.
Denote by $\cG_\cP:= \Aut_{\cG}(\cP/C)$ the group scheme of automorphisms of $\cP$ over
$C$, which is \'etale locally isomorphic to $\cG$. In particular this is again a simply connected Bruhat-Tits group scheme over $C$, which is of the same type as $\cG$. Furthermore $\Bun_{\cG_\cP}\cong
\Bun_\cG$, the isomorphism being given by $\cQ \mapsto \cQ\times^{\cG_\cP} \cP$ (here $\cG_P$ acts on the right on $\cQ$ and on the left on $\cP$, commuting with the right action of $\cG$ on $\cP$, so that the quotient $\cQ\times \cP/\cG_\cP$ is a (right) $\cG$ torsor on $C$).
Consider the map $\pr_x^{\cP}\colon  \Gr_{\cG_\cP,x} \to \Bun_{\cG_\cP} \cong \Bun_\cG$. Its
image consists of those $\cG$-bundles which are isomorphic to $\cP$ over $C-x$. We
already know that the image of this map is open and that it is contained in the connected
component of $\cP \in \Bun_\cG$. Thus there exists a bundle $\cP^\prime\in \im(\pr_x)\cap
\im(\pr_x^\cP)$, but that means that $\cP|_{C-x} \cong \cP^\prime|_{C-x} \cong
\cG|_{C-x}$. Thus $\cP\in \pr_x$ and therefore all the maps $\pr_y$ have the same image. Moreover we have shown that given two $\cG$-bundles $\cP,\cP^\prime$ such that $\cP|_{C-y}\cong \cP^\prime|_{C-y}$ for some $y\in C$ then $\cP|_{C-x}\cong \cP^\prime|_{C-x}$ for all $x\in C$.

By the theorem of Steinberg and Borel--Springer we know that any bundle is trivial on some open subset $U\subset C$ and can thus be obtained by gluing trivial bundles on $U\subset C$ and on the formal completions of the finitely many points $C-U$.

However we have just seen, that gluing at different points does not
produce new bundles: Namely fix $x\in C$ and suppose that $\cP$ is given by gluing the trivial bundle on $U=C-\{x_1,\dots,x_n\}$ via the gluing functions $g_i\in L\cG_{x_i}(\kbar)$ for $i=1,\dots n$. Let $\cP_j$ be the bundle given by ${g_1,\dots,g_j}$ for $0\leq j\leq n$, in particular $\cP_0$ is the trivial bundle. Then we know that $\cP_j|_{C-x_j}\cong \cP_{j-1}|_{C-x_j}$ and we have shown that this implies that $\cP_j|_{C-x}\cong \cP_{j-1}|_{C-x}$. Therefore $\cP|_{C-x}\cong \cP_0|_{C-x}$. And therefore the map $\pr_x$ is surjective on $k$-valued points.

Thus Corollary \ref{Glatte_Umgebungen} can be applied to every point in the image of $S\to \Bun_\cG$,
in order to obtain a smooth covering of $S$, factoring through $\Gr_x$. Since every smooth covering has an \'etale refinement this is sufficient to prove our theorem.
\end{proof}
\begin{remark}
If $\cG$ is simply connected the affine flag manifold $\Gr_{\cG,x}$ is connected
(\cite{Pappas-Rapoport} Theorem 0.1). Therefore Theorem \ref{EinfachzusammenhaengendeUniformisierung} and Proposition
\ref{Uniformisierungs_Abbildung} (1) imply that $\Bun_\cG$ is connected if $\cG$ is
simply connected. This proves Theorem \ref{Zusammenhangskomponenten} for simply connected groups $\cG$.
\end{remark}

\section{The case of general groups}

In this section we want to deduce the general case of the uniformization theorem from the case of simply
connected groups. As in the previous section the formulation for group schemes is helpful in order to circumvent the reduction to Borel subgroups used in \cite{Drinfeld-Simpson}.

Since the statement of the uniformization theorem allows to pass to finite extensions of $k$ we may assume that $C$ is defined over an algebraically closed field $k$.

Let us begin with some preparations concerning simply connected coverings of Bruhat-Tits group schemes.
Let $\cG$ be a Bruhat-Tits group scheme over $C$ and denote the open subset of $C$ over which $\cG$ is semisimple by $U$. Then we know that over $U$ there exists a simply connected covering $p_U\colon \tilde{\cG}_U\to \cG_U$, which is a finite morphism. We denote the kernel $\cZ_U:=\ker(p_U)$. Furthermore, since $k$ is algebraically closed, the group schemes $\cG_U,\widetilde{\cG}_U$ are quasi-split.

To extend this to the whole of $C$, we consider the local problem around $x\in \Ram(\cG)$ and then glue the result, as in the proof of Proposition \ref{Uniformisierungs_Abbildung}. In the local situation we obtain a canonical extension of $\widetilde{\cG}_U$, because the Bruhat-Tits building of a group $G$ over a local field is isomorphic to the building of the simply connected of $G$ (\cite{Landvogt} 2.1.7). To use this abstract result, we need to recall some points of the construction of Bruhat-Tits.
 
Write $R:=\widehat{\cO}_{C,x}$ and $K:=K_x$ for the local field at $x$. Since $\cG_K$ is quasi-split $\cG_R$ is obtained by first extending a maximal torus $T_K\subset \cG_K$ by taking its connected N\'eron model and then extending the root subgroups $U_a\subset \cG_K$ according to the choice of a facet in $X_*(T^0_K)$, where $T^0_K\subset T_K$ is a maximal split torus in $T_K$.

The choice of $T^0_K,T_K$ determines tori $\tilde{T}^0_K,\tilde{T}_K$ in $\tilde{\cG}_K$, and we can use the same extensions of the root subgroups $U_a$ to construct $p_R\colon \tilde{\cG}_R\to \cG_R$. By construction, the kernel of $p_R$ is the kernel $\cZ$ of the map on the maximal tori $p\colon \tilde{\cT}\to \cT$. We will need some control over the group scheme $\cZ$. 

Since $\tilde{\cG}_K$ is simply connected, the torus $\tilde{T}_K$ is an induced torus, i.e. there exists a generically \'etale (possibly disconnected) extension $\Spec R^\prime \to \Spec R$ such that $\tilde{T}_K$ is the Weil-restriction of $\bG_m$ over $\Spec K\times_{\Spec R}\Spec R^\prime$. Moreover in this case the connected N\'eron model $\tilde{\cT}$ of $\tilde{T}_K$ is the Weil restriction to $\Spec R$ of $\bG_m$ on $\Spec R^\prime$ (\cite{BT2} 4.4.8). 

\begin{claim}\label{Zfin}
The closure $\cZ^{\fin}$ of $\cZ_K$ in $\tilde{\cT}$ is a finite flat group scheme over $\Spec(R)$.
\end{claim}

\begin{proof}
Since the center $\cZ_K$ is finite, it is contained in the $n-$torsion $T_K[n]\subset T_K$ for some $n$ and it is sufficient to show that the closure of $T_K[n]$ in $\cT$ is finite over $R$. It is certainly quasi-finite and to check properness, we observe that $T_K[n]$ is the Weil restriction of $\mu_n=\bG_m[n]\subset \bG_m$ over $\Spec(R^\prime)$. So if $L\supset K$ is the quotient field of a discrete valuation ring $O$, then an $L$-point of $T_K[n]$ is a $L\tensor_R R^\prime$ point of $\mu_n$. Since $\mu_n$ is proper, this extends canonically to an $O\tensor_R R^\prime$ point of $\mu_n$, which defines an $O$-point of the Weil restriction.
\end{proof}

\begin{remark}
If $\Spec R^\prime \to \Spec R$ is ramified and the characteristic of the ground field divides $n$, then the Weil restriction of $\mu_n$ on $R^\prime$ is not finite over $R$. In particular it can happen that the group scheme $\cZ$ is only generically finite. However we also see that $\cZ=\cZ^{\fin}$ is finite and flat if either $R^\prime$ is an unramified extension of $R$ or the characteristic of the base field does not divide the order of $\pi_1(\cG)$.
\end{remark}

In order to make some global computations, we note that $\cG_U$ being quasi-split implies that we can find a Borel subgroup $\cB_U\subset \cG_U$ so the quotient of $\cB$ by its unipotent radical is a torus $\cT_U$ over $U$. We can consider its connected N\'eron model $\cT$ over $C$ and the same holds for the simply connected covering $\tilde{\cG}$, so we find an exact sequence $0\to \cZ\to \tilde{\cT}\to \cT$.
Here, we need to note that $\cT$ need not be contained in $\cG$. However, since the maximal tori of $\cG_{K_x}$ are conjugated, locally around any point $x\in C$ the torus $\cT_{\widehat{\cO}_{C,x}}$ is isomorphic to the torus used in the Bruhat-Tits construction. Furthermore, since $\cZ_R$ is contained in the center of $\tilde{\cG}_R$ the kernel $\cZ$ of $\tilde{\cT}\to \cT$ is independent of the choice of the maximal torus over $R$, so that the kernel of $\tilde{\cT}\to \cT$ is indeed isomorphic to the kernel of $\widetilde{\cG}\to \cG$.

Denote by $\cT^{\fin}:=\tilde{\cT}/\cZ^{\fin}$, which is a smooth group scheme, generically isomorphic to $\cT$ and the sequence 
\begin{equation}\label{fin}
0 \to \cZ^{\fin} \to \tilde{\cT} \to \cT^{\fin} \to 0
\end{equation}
defines an exact sequence of fppf-sheaves of groups. In particular this induces a long exact sequence of cohomology groups.

Similarly we define $\cG^{\fin}:=\tilde{\cG}/\cZ^{\fin}$ and again we get a short exact sequence:
$$ 0 \to \cZ^{\fin} \to \tilde{\cG} \to \cG^{\fin} \to 0.$$

Moreover we claim:
\begin{lemma}\label{Zfin_lemma}
\begin{enumerate}
\item The natural morphism $H^1(C,\cT^{\fin}) \to H^1(C,\cT)$ is surjective. The map $\Bun_{\cT^{\fin}} \to \Bun_\cT$ is smooth, surjective, with connected fibres.
\item Given a family $\cP$ of $\cT$-torsors, parametrized by a connected scheme $S$, the obstruction to lift $\cP$ to a family of $\widetilde{\cT}$-torsors fppf-locally on $S$ is constant on $S$.
\item Let us denote $\overline{H^2(C,\cZ^{\fin})}:=H^2(C,\cZ^\fin)/H^0(C,\cT/\cT^\fin)$, where $H^0(C,\cT/\cT^{fin})\to H^1(C,\cT^{fin})\to H^2(C,\cZ^\fin)$ is a composition of boundary maps.
Then there is an exact sequence $H^1(C,\tilde{\cT}) \to H^1(C,\cT) \to \overline{H^2(C,\cZ^{\fin})}$.
\item Parts (1)-(3) also hold if we replace $\cT,\cT^{\fin},\tilde{\cT}$ by $\cG,\cG^{\fin},\tilde{\cG}$.
\end{enumerate}
\end{lemma}
\begin{proof}
We know that for any torus $T$ over $k(C)$ we have $H^1(k(C),T)=0$ (\cite{SerreCohomologieGaloisienne} Corollaire p.170). So any $\cT-$torsor is generically trivial and thus any $\cT$-torsor lies in the image of the gluing map $\oplus_{x\in C} L\cT_{x_i}\to \Bun_{\cT}$. Since $L\cT_{x_i}\cong L\cT^{\fin}_{x_i}$ this proves the claimed surjectivity. The map $\cT^\fin \to \cT$ induces a map on the Lie algebras $\Lie(\cT^\fin)\to \Lie(\cT)$ which is an isomorphism on the generic fibre. In particular, the map $H^1(C,\Lie(\cT^\fin))\to H^1(C,\Lie(\cT))$ is surjective, and thus the map $\Bun_{\cT^\fin}\to \Bun_\cT$ is smooth. 
Finally, the elements of the kernel of $H^1(C,\cT^{\fin})\to H^1(C,\cT)$ are obtained from elements in $$\oplus_{x\in \Ram(\cT)} L^+\cT_x(k)/L^+\cT^{\fin}_x(k)\subset \oplus_{x\in \Ram(\cT)}L\cT^{\fin}_x(k)/L^+\cT^\fin_x(k).$$ Since $L^+\cT$ is a connected scheme, this is connected. This proves part (1).

Part (1) for $\cG$ follows by the same argument using the theorem of Steinberg and Borel--Springer as recalled in section \ref{local_triviality}.

To show (2) we want to prove that the map $\Bun_{\tilde{\cT}}$ to $\Bun_{\cT}$ is flat with finite fibres. 

As in Claim \ref{Zfin}, there is an $n>0$ such that the multiplication by $n$ on $\tilde{\cT}$ factors through the map ${\tilde{\cT}}\to {\cT}$. Since multiplication by $n$ has finite fibres on $\Bun_{\tilde{\cT}}$, the map $\Bun_{\tilde{\cT}}\to \Bun_{\cT}$ has finite fibres as well.

Next, we recall that the dimension of the stacks $\Bun_{\tilde{\cT}}$ and $\Bun_{\cT}$, which are both smooth, only depends on the rank and degree of the Lie algebras $\Lie(\tilde{\cT})$ and $\Lie(\cT)$. These sheaves have the same degree by the the main theorem of \cite{ChaiYu}, so the two stacks have the same dimension. 

We claim that this implies that the canonical map $\Bun_{\tilde{\cT}}\to \Bun_{\cT}$ is flat. First, since both stacks carry a group structure it is sufficient to show that the map on the connected component of the identity $p:\Bun_{\tilde{\cT}}^\circ \to \Bun_\cT^\circ$ is flat on some non-empty open substack of $\Bun_{\tilde{\cT}}^\circ$.

To show this, take a smooth, connected scheme $X$ together with a smooth dominant map $p:X\to \Bun_\cT^\circ$ of relative dimension $d_p$. Also take a smooth dominant map $q:Y\to X\times_{\Bun_\cT^\circ} \Bun_{\tilde{\cT}}^\circ=:F$ of relative dimension $d_q$. In particular the map $Y \to \Bun_{\tilde{\cT}}^\circ$ is smooth of relative dimension $d_q+d_p$.
If the induced map $f:Y\to X$ was not dominant, then $f$ had a fibre of dimension $>d_q$. This implies that $F\to X$ has a fibre of dimension $>0$, which cannot happen since the fibres of $\Bun_{\tilde{\cT}} \to \Bun_\cT$ are finite.
Since $X$ is integral and $f$ is dominant, generic flatness implies that the map $Y \to X$ is flat on some nonempty open $U\subset Y$. This implies that $p$ is generically flat, because flatness can be checked on smooth coverings. This shows our claim.

In particular, for any family of $\cT$-torsors parametrized by a connected scheme $S$, the obstruction to lift the family fppf-locally on $S$ to a family of $\tilde{\cT}$-torsors is constant on $S$. 

Statement (3) follows from the exactness of the sequences 
$$ H^1(C,\tilde{\cT})\to H^1(C,\cT^\fin) \to H^2(C,\cZ^\fin)$$
and 
$$ H^0(C,\cT/\cT^\fin) \to H^1(C,\cT^\fin) \to H^1(C,\cT).$$
Here the second sequence is obtained from the map $\cT^\fin \to \cT$ which induces an injection of fppf-sheaves on $C$, since the map is an isomorphism on the generic fibre. 

To deduce the corresponding statements for $\cG$-torsors we note that the sheaves $\cG/\cG^\fin\cong \cT/\cT^\fin$ are isomorphic. This is because the $\cG$ and $\cG^\fin$ contain open subsets isomorphic to $\cU \cT \cU^-$ and $\cU \cT^\fin \cU^-$ respectively, generating the connected group schemes $\cG$ and $\cG^\fin$. We have already seen that after passing to a smooth neighbourhood any family of $\cG$-torsors can be lifted to a family of $\cG^\fin$ torsors. 

The obstruction to lift a family of $\cG^\fin$ torsors to a family of $\widetilde{\cG}$-torsors locally on the base $S$ is given by a class in $\bR^2\pr_{S,*}\pr_C^* \cZ^{\fin}$.  So the obstruction to lift a family of $\cG$-torsors to a ${\widetilde{\cG}}$-torsor is given by an element in $\bR^2\pr_{S,*}\pr_C^* \cZ^{\fin}/\bR^0 \pr_{S,*} \pr_C^* \cG/\cG^\fin$. Since $\cG/\cG^\fin$ is supported at $\Ram(\cG)$ this quotient is the same as the corresponding quotient in the case of $\cT$-torsors (see Remark \ref{lokale Beschreibung der Liftungsgerbe} for a local description of the obstruction classes in $H^2(C,\cZ^\fin)$). This already proves (3) for $\cG$.

We are left to show that the above quotient sheaf is locally constant. In the following Lemma we will see that $H^1(C,\cT^\fin) \to H^2(C,\cZ^\fin)$ is surjective. Since we already know that the obstruction to the existence of a lift of a $\cT$-torsor to a $\tilde{\cT}$-torsor is locally constant, this implies that the quotient $\bR^2\pr_{S,*}\pr_C^* \cZ^{\fin}/\bR^0 \pr_{S,*} \pr_C^* \cT/\cT^\fin$ is locally constant and it is isomorphic to the above quotient sheaf.
\end{proof}

Knowing that the obstruction to lift a given $\cG$-torsor to a $\tilde{\cG}$-torsor is given by an element in $\overline{H^2(C,\cZ^{\fin})}$ we will need to calculate the latter group. This we can do, because we have realized $\cZ$ as a subgroup of an induced torus i.e., a groups of the form $\text{Res}_{D/C}\bG_m$ as considered in Example (3) in the introduction, where $D\to C$ is a finite, generically \'etale, but possibly disconnected covering of $C$.

\begin{lemma}\label{Z_in_induziertem_torus} Assume that $k=\overline{k}$ is algebraically closed. Then $H^2(C,\tilde{\cT})=0$ and therefore the sequence:
$$0 \to \cZ^{\fin} \to \tilde{\cT} \to \cT^{\fin} \to 0$$
defines an exact sequence
$$H^1(C,\tilde{\cT}) \to H^1(C,\cT^{\fin}) \to H^2(C,\cZ^{\fin})\to 0.$$	
By Lemma \ref{Zfin_lemma} this induces an exact sequence:
$$H^1(C,\tilde{\cT}) \to H^1(C,\cT) \to \overline{H^2(C,\cZ^{\fin})}\to 0.$$	
\end{lemma}
\begin{proof}
Write $\tilde{\cT}=\pi_*(\bG_m)$, where $\pi\colon D\to C$ is a finite covering. We claim that the Leray spectral sequence $$H^*(C,\bR\pi_* \bG_m)\Rightarrow H^2(D,\bG_m)$$ defines an isomorphism $H^2(C,\tilde{\cT})=H^2(D,\bG_m)=0$, because the higher derived images vanish in the fppf-topology.
This holds because $R^1\pi_*(\bG_m)$ is the sheaffification of the group of line bundles on the finite fibres of $\pi$, and $H^2$ of the fibres classifies gerbes, which are also calculated by the corresponding \'etale cohomology, which is $0$.

Thus we obtain $$H^1(C,\tilde{\cT}) \to H^1(C,\cT^{\fin}) \to H^2(C,\cZ^{\fin})\to 0.$$
\end{proof}

In order to make this more explicit we need to check that the moduli space of torsors under a torus has the expected number of connected components. Recall that for a torus $\cT$ over $k(C)$ the fundamental group is $\pi_1(\cT):=X_*(\cT_{k(C)^{sep}})$, considered as a $\Gal(k(C)^{sep}/k(C))$-module.
\begin{lemma}\label{pi0_Tori}
Keeping the notation from Lemma \ref{Z_in_induziertem_torus} we have:
\begin{enumerate}
 \item There is an isomorphism $\pi_0(\Bun_\cT) \map{\cong} X_*(\cT_{k(C)^{\textrm{sep}}})_{\Gal(k(C)^{\textrm{sep}}/k(C))}$.
 \item There is an exact sequence:
$$ \pi_1(\tilde{\cT})_{\Gal(k(C)^{\textrm{sep}}/k(C))} \to \pi_1(\cT)_{\Gal(k(C)^{\textrm{sep}}/k(C))} \to \overline{H^2(C,\cZ^{\fin})}\to 0.$$
In particular $\pi_1(\cG)_{\Gal(k(C)^{\textrm{sep}}/k(C))}\cong \overline{H^2(C,\cZ^{\fin})}$.
\end{enumerate}
\end{lemma}
\begin{proof}
If $\cT=\pi_*\bG_m$ is an induced torus, then (1) holds: In this case the connected components of $\Bun_{\cT}=\Pic_D$ are given by the degrees of line bundles on the connected components of $D$ and $X_*(\cT)=\pi_*(\bZ)$, as sheaves over $C$. 

Next, let $\cT$ be arbitrary. As in the definition of the Kottwitz homomorphism (cf. \cite{Pappas-Rapoport} section 3 for a brief review of the construction of the Kottwitz homomorphism) we choose induced tori $\cI_2 \to {\cI_1} \tto \cT$ such that the induced sequence  $X_*(\cI_2) \to X_*(\cI_1)\to X_*(\cT)\to 0$ is exact. By \cite{Pappas-Rapoport} Theorem 5.1 we know that for any $x\in C$ we have $\pi_0(L\cT_x)\cong X_*(\cT)_{\Gal({K_x}^{\textrm{sep}}/K_x)}$. Thus we obtain an exact sequence $$\pi_0(L\cI_{2,x}) \to \pi_0(L\cI_{1,x}) \tto \pi_0(L\cT_x)\to 0.$$ Since any $\cT$-torsor lies in the image of the gluing map $\bigoplus_{x\in C} \cT(K_x) \to \Bun_{\cT}(k)$ we obtain an exact sequence $$\pi_0(\Bun_{\cI_2}) \to \pi_0(\Bun_\cI) \tto \pi_0(\Bun_{\cT})\to 0.$$ This implies (1). 

Finally (2) follows applying $\pi_0(\underline{\quad})$ to the last sequence in Lemma \ref{Z_in_induziertem_torus} and then using (1).
\end{proof}

\begin{remark}\label{lokale Beschreibung der Liftungsgerbe}
The above lemma shows that the surjection $\cT(K_x)=\cT^{\fin}(K_x) \tto {H^2(C,\cZ^{\fin})}$ can be obtained by associating to a $\cT^{\fin}$-torsor the gerbe of liftings to a $\tilde{\cT}$-torsor. Alternatively we can describe this using the diagram:
$$\xymatrix{
H^1(C, \cT^\fin) \ar[r]       & H^2(C,\cZ^{\fin})\\
\cT(K_x)    \ar[r]\ar[u] & H^1(K_x,\cZ^{\fin})\ar[u]\\
}$$
Namely, given a $\cZ^{\fin}$-torsor $\cP\to \Spec(K_x)$ we can define a $\cZ^{\fin}$-gerbe by the groupoid:
$$\xymatrix{ \cZ^{\fin}|_{\Spec(\Ox)} \cup \cZ^{\fin}|_{C-x}\cup \cP \dar &\Spec(\Ox) \cup (C-x)},$$
in which the source and target morphisms for the first two spaces are the projections and for $\cP$ the source morphism is the projection to $\Spec(\Ox)$ and the target morphism is the projection to $(C-x)$, composition is given by multiplication and the $\cZ$-torsor structure of $\cP$. To prove that this defines an algebraic stack we only need to note that any torsor over $K_x$ extends to a scheme of finite type over $C-x$. 

We claim that this gerbe has a natural morphism to the lifting gerbe. The lifting gerbe of a $\cT$-torsor $\cQ_0$ is the stack which is given by associating to any flat $f\colon S\to C$ the category of $\tilde{\cT}$-torsors $\cQ$ over $S$ together with an isomorphism of the $\cT$-torsors $\cQ\times^{\tilde{\cT}}\cT\cong f^*(\cQ_0)$. This is a $\cZ^{\fin}$-gerbe, because for any flat $S\to C$ sections of $\tilde{\cT}(S)$ mapping to $1$ in $\cT(S)$ automatically factor through $\cZ^{\fin}$. Moreover this gerbe is neutral over $\Spec(\widehat{\cO}_{C,x})$ and $(C-x)$ because the trivial $\cT$-torsors admits a reduction to $\tilde{\cT}$.

Given $g\in \cT(K_x)$ denote by $\cQ_g$, the associated $\cT$-torsor over $C$ and by $\cP_g$ the $\cZ^{\fin}$-torsor over $\Spec(K_x)$. Finally denote the lifting gerbe of $\cT$ by $\cZ_\cT$. The gerbe constructed above maps to the lifting gerbe, because $\cP_g$ is canonically trivial over $\Spec(\widehat{\cO}_{C,x})$ and $C-x$ so we obtain canonical morphisms $\Spec(\widehat{\cO}_{C,x})\to \cZ_{\cT}$ and $(C-x)\to \cZ_{\cT}$. The difference betwen the two trivializations of $\cP_g$ over $K_x$ is given by $g$, so $\Spec(\widehat{\cO}_{C,x}) \times_{\cZ_\cT} (C-x) \cong \cP_g$. Thus we obtain the claimed morphism of $\cZ^{\fin}$-gerbes, but any morphism of $\cZ^{\fin}$-gerbes is an isomorphism.

This description has the advantage that it only uses the local structure of $\cT$ at $K_x$, which allows us to compare the construction for groups which are only locally isomorphic to $\cT$.
\end{remark}

Now we can prove the uniformization theorem stated in the introduction:
\begin{theorem}\label{DS}
Let $\cG$ be a parahoric Bruhat-Tits group scheme over $C$ such that the generic fiber of
$\cG$ is semisimple. Then for every point $x\in C$, every noetherian scheme $S$ and every
family $\cP \in \Bun_\cG(S)$ there exists an fppf-covering $S^\prime \to S$ such that
$\cP|(C-x)\times S^\prime$ is trivial. In particular the map $\Gr_{\cG,x} \to \Bun_\cG$
is a formally smooth, surjective map of stacks.
\end{theorem}

\begin{proof}
Denote by $\tilde {\cG}$ the simply connected covering of $\cG$. We have seen that there is an exact sequence:
$$ H^1(C,\tilde{\cG})\to H^1(C,\cG) \to \overline{H^2(C,\cZ^{\fin})}\to 0$$
defined by the obstruction to lift a $\cG$-torsor to a $\tilde{\cG}$-torsor. Furthermore  we have seen that the  class in $\overline{H^2(C,\cZ^{\fin})}$ is locally constant. This implies that every class $d\in \overline{H^2(C,\cZ^{\fin})}$ defines an open and closed substack $\Bun_\cG^d$.

Thus we may assume that $S$ is connected and maps into $\Bun_\cG^d$ for
some $d$. If $d=0$ we may argue as before: The obstruction to lift our $\cG$-torsor $\cP$ to
$\tilde{\cG}$ vanishes after passing to a covering $S^\prime\to S$. Thus we
can find $S^\prime\to S$ such that $\cP$ lifts to a $\tilde{\cG}$-torsor $\tilde{\cP}$
on $C\times S^\prime$ and for $\tilde{\cG}$-torsors we can apply Theorem \ref{EinfachzusammenhaengendeUniformisierung} (the uniformization theorem).

Now if $d$ is arbitrary we can apply a similar argument to reduce the question to the
situation where $S$ is a geometric point: Assume again that $S$ is connected and that
$\cP$ is a family of $\cG$-bundles on $S$. Choose a point $0\in S$ and consider the
group scheme $\cG_0:=\Aut_{\cG}(\cP_0)$ over $C$. The simply connected covering $\tilde{\cG}_0$ again defines a sequence
$\cZ \to \tilde{\cG}_0 \to \cG_0$. (Here the kernel is given by the same group scheme $\cZ$ as before, because $\cG_0$ is an inner form, i.e., it is in the image $H^1(C,\cG)\to H^1(C,\Aut(\cG))$ and inner automorphisms act trivially on the center.)

The trivial $\cG_0$-torsor $\cP_0$ certainly lifts to
a $\tilde{\cG}_0$-torsor (see e.g. \cite{Giraud} Prop. 4.3.4) and so we can apply the
same reasoning as before to see that the $\cG_0$-torsor $\cP\times^{\cG} \cP_0$ lifts to
a $\tilde{\cG}_0$-torsor locally on $S$. Passing to a further covering we therefore 
find that $\cP|_{(C-x)\times S^\prime} \cong \cP_0\times S^\prime|_{(C-x)\times
S^\prime}$. Thus it is sufficient to show that the $\cG$-torsor $\cP_0$ is trivial on $C-x$.

To prove this, we show that one can modify $\cP_0$ at the point $x$, such that the obstruction $d$ vanishes. 
This is implied by the surjectivity of the Kottwitz homomorphism: The isomorphism $\pi_0(L\cG_x) \to \pi_1(\cG_{\overline{K}_x})_{\Gal({K}_x^{\textrm{sep}}/K_x)}$ is defined by a reduction to tori and thus we can apply Lemma \ref{pi0_Tori} and Remark \ref{lokale Beschreibung der Liftungsgerbe} to conclude that $\pi_0(L\cG_x) \to \overline{H^2(C,\cZ^{\fin})}$ is surjective.
\end{proof}
As a corollary of the above proof we can also verify the conjecture on the connected components of $\Bun_\cG$:
\begin{theorem}
Let $\cG$ be a quasi-split semisimple parahoric Bruhat-Tits group scheme. Then
$$\pi_0(\Bun_\cG) = \pi_1(\cG_{\overline{\eta}})_{\Gal({k(\eta)}^{\textrm{sep}}/k(\eta))}.$$
\end{theorem}
\begin{proof}
In the above proof we have seen that for every $d\in \overline{H^2(C,\cZ^{\fin})}$ the stack $\Bun_\cG^d$ is connected and nonempty, and in Lemma \ref{pi0_Tori} we have seen that $\overline{H^2(C,\cZ^{\fin})}\cong \pi_1(\cG_{\overline{\eta}})_{\text{Gal}({k(\eta)}^{\textrm{sep}}/k(\eta))}$. 

Without reference to the proof above we can use the global Gra\ss mannian $\GR_\cG\to C$ to get an alternative argument:

Again denote by $U\subset C$ an open subset such that $\cG|_U$ is semisimple. Then the surjection $\GR_\cG|_U \tto \Bun_\cG$ induces a surjection $\pi_0(\GR_\cG|U) \tto \pi_0(\Bun_\cG).$ We claim that $\pi_0(\GR_\cG|_U) \cong \pi_1(\cG_{\overline{\eta}})_{\text{Gal}({k(\eta)}^{\textrm{sep}}/k(\eta))}$. To see this, observe that the formation of $\GR$ commutes with \'etale base change. Thus to compute the \'etale sheaf on $U$ given by the connected components of the fibres $\GR_\cG\to C$, we may (after possibly shrinking $U$) reduce to the case that $\cG$ is a split group scheme. In this case the connected components are canonically isomorphic to $\pi_1(\cG)$ by the Kottwitz homomorphism on the fibres (this morphism is constant, since the definition of the Kottwitz homomorphism uses a reduction to the case of tori). In particular the connected components of $\GR_\cG|_U$ are given by the orbits of the Galois group on $\pi_1(\cG)$.

The inverse map $\pi_0(\Bun_\cG) \to \pi_1(\cG_{\overline{\eta}})_{\text{Gal}({k(\eta)}^{\textrm{sep}}/k(\eta))}$ is given by the obstruction class in $\overline{H^2(C,\cZ^\fin)}$ and Lemma \ref{pi0_Tori}.
\end{proof}

\section{Line bundles on $\Bun_\cG$}

Rapoport and Pappas conjectured that for simply connected and absolutely simple groups $\cG$ splitting over a tamely ramified extension of $k(C)$ there should be an exact sequence:
$$ 0 \to \prod_{x\in \text{Bad}(\cG)} X^*(\cG_x) \to \Pic(\Bun_\cG) \map{c} \bZ \to 0$$
where the map $c$ should be given by the so called central charge, defined for any point $x\in C$ as follows (see \cite{Pappas-Rapoport} Remark 10.2): First, the map $\Gr_{\cG,x}\to \Bun_\cG$ defines a map $\Pic(\Bun_\cG) \to \Pic(\Gr_{\cG_x})$ and there is a canonical morphism $\Pic(\Gr_{\cG,x})=\bZ^{N_x} \to \bZ$ which can be described explicitly in terms of the root datum of $\cG_x$. In the case of constant groups $G$ this morphism is usually described in terms of a central extension of $LG$. Namely the obstruction to lift the action of $LG$ to the line bundles on $Gr_{G,x}$ defines a central extension $\widetilde{LG}$ of $LG$ and the central charge homomorphism is defined by the weight of the action of the central $\bG_m\subset \widetilde{LG}$ on the line bundle. A similar description also holds in the general case, since we will see that the obstruction to the existence of an $L\cG_x$-linearization of the line bundles on $\Gr_{\cG,x}$ only depends on its central charge, so again there exists one central extension of $L\cG_x$ acting on all line bundles on $\Gr_{\cG,x}$.

In this section we will denote by $\Pic(\Bun_{\cG})$ the group of line bundles, rigidified by the choice of a trivialization over the trivial $\cG$-torsor. This is useful in order to compare the Picard groups of schemes mapping to $\Bun_{\cG}$.

The assumptions on $\cG$ will be used in our proofs, since we will apply the computation of $\Pic(\Gr_{\cG,x})$ given in \cite{Pappas-Rapoport} and we will also need the fact that $\Gr_{\cG,x}$ is reduced and connected.

First of all, note that there exists a natural morphism $$\prod_{x\in \text{Bad}(\cG)} X^*(\cG_x) \incl{} \Pic(\Bun_\cG),$$  which can be constructed as follows: For any point $x\in C$ restriction to $x$ defines a morphism $\Bun_\cG \to B\cG_x$. Now $\Pic(B\cG_x)\cong X^*(\cG_x)$ by definition, since a coherent sheaf on $B\cG_x$ is the same thing as a representation of $\cG_x$. Thus we get a pull-back map $\prod_{x\in \text{Bad}(\cG)} X^*(\cG_x) \to \Pic(\Bun_\cG)$.

To see that this morphism is injective, consider the composition $\Gr_y \to \Bun_\cG \to B\cG_x$. By Lemma \ref{Uniformisierungs_Abbildung}, $\Gr_y$ is the classifying space for torsors on $C$ together with a trivialization on $C-y$. In particular for $y\neq x$ the above map is given by the  trivial $\cG_x$-torsor and therefore induces the $0$-map on the Picard groups.

Furthermore we claim that for $x=y$ the map $\Gr_x \to B\cG_x$  defines an injection $X^*(\cG_x) \incl{} \Pic(\Gr_x)$. To simplify notations let us assume that $L^+\cG(k)\subset L\cG_x(k)$ is an Iwahori subgroup (otherwise we can find a smaller Bruhat--Tits group scheme $\cG^\prime \to \cG$ and check injectivity after pulling back everything to $\Bun_{\cG^\prime}$).
Recall that for any affine simple root $\alpha$ of $L\cG_x$ we get an embedding $i_\alpha\colon \bP^1\to \Gr_x$ coming from an embedding of a parahoric group $L^+P_i \to L\cG$. This way we get a commutative diagram:
$$\xymatrix{
L^+P_i \ar[r]\ar[d] & L\cG_x \ar[r]\ar[d] & \text{\rm pt} \ar[d]\\
\bP^1=L^+P_i/L^+\cG \ar[r]& \Gr_x\ar[r] & B\cG_x .\\
}$$
So we see that the degree of the restriction of the bundle given by a character $\lambda$ on $\bP^1$ is given by the restriction of this character to the root-subgroup $\bG_m \to P_i$. This constructs the left hand side of the exact sequence of the statement of Theorem \ref{PicBunG}.

For the right hand side we would like to argue as in the case of constant groups \cite{Faltings_Loopgroups}, but in order to take care of the ramification of $\cG$ we need the relative affine Gra\ss maninan $\GR_\cG$:

\begin{lemma} Assume that $k$ is algebraically closed, $\cG_{k(C)}$ is simply connected and absolutely simple and splits over a tamely ramified extension of $k(C)$.
\begin{enumerate}
\item The relative Gra\ss mannian $\GR_{\cG}\to C$ is ind-proper (this holds for general $\cG$).
\item The relative Picard group $\Pic(\GR_\cG/C)\to C$ is an \'etale sheaf over $C$ and there exists a quotient $c:=\Pic(\GR_\cG/C)/\prod X^*(\cG_x)$. 
\item The fibres of $c$ are isomorphic to $\bZ$ and the restriction of $c$ to $C-\Ram(\cG)$ is constant.
\end{enumerate}
\end{lemma}
\begin{proof}
For the first part we only have to recall that we constructed $\GR_\cG$ as a closed subscheme of an affine Gra\ss mannian for the constant group $\GL_n$, which is ind-projective.

To prove (2), we first consider this Zariski-locally over open subsets $U\subset C$, as in the construction of $\GR_\cG$. In particular we may assume that there is a function $t$ on $U\times U$, which is a local parameter along the diagonal. We want to write $\GR_\cG=\varinjlim Z_i$ where $Z_i\subset Z_{i+1} \subset \cdots$ is a chain of closed embeddings of proper, connected, reduced schemes, flat over $U$. 

To get proper, connected schemes, we want to use Schubert varieties, so we need to find a torus $\cT\subset\cG$ over $U$: 
Since $\cG_{k(C)}$ is quasi-split, we can choose $\cT_V\subset \cG_V$ for a small enough $V\subset U$. Denote by $\cT$ the connected N\'eron model of $\cT_V$ over $U$. For every $x\in U-V$ we know that all maximal tori contained in some a Borel-subgroup are conjugated, so $\cT_{K_x}$ is conjugated to the torus used to construct the Bruhat-Tits group $\cG$. Thus the elements in $\cG(K_x)$ needed to conjugate $\cT_{K_x}$ define a $\cG$-torsor $\cP$ such that for $\cG_{\cP}=\Aut_{\cG}(\cP)$ we have $\cT\subset \cG_\cP$. By the uniformization theorem we know that there exists a trivialization of $\cP$ over every open subset $U\subsetneq C$. So we can find an isomorphism $\cG_\cP|_U \cong \cG_U$ and thus an inclusion $\cT\subset \cG|_U$. 

Moreover, $\cT$ being an induced torus we can write $\cT=\Res_{D/U} \bG_m$ for some finite, generically \'etale covering $D\to U$ and we have $\bG_{m,U}^{\pi_0(D)}\subset \cT \subset \cG$. Also write $\mathring{U}:=U-\Ram(\cG)$.
We can already deduce that $\GR_\cG$ is the closure of $\GR_{\cG,\mathring{U}}$, because on the one hand all fibres are reduced by \cite{Pappas-Rapoport} and we claim that all geometric points of a special fibre $x\in \Ram(\cG)$ lie in the closure.
To see this last point, note that we have seen that any geometric point of $\Gr_{\cG_x}$ can be lifted to an element of $L\cG_x$ (Proposition \ref{LG} (3)) and these elements can be written as products of elements in root subgroups $L\cU_{a,x}$. The torus $\cT$ defines root subgroups $\cU_a$ over $U$ and as varieties these are locally trivial bundles of affine spaces over $U$, so that any element in $L\cU_{a,x}(k)$ can be extended to a local section of $L\cU_a$. The product of these elements gives the extension we were looking for. In particular, given any presentation $\GR_{\cG,\mathring{U}}=\varinjlim Z_i$ as an inductive limit of closed subschemes, the limit of the closures of the $Z_i$ in $\GR_{\cG,U}$ will be $\GR_{\cG,U}$.

To define global Schubert varieties, note that our choice of a local parameter defines for every $w\in X_*(\bG_m^{\pi_0(D)})$ a point in $L\cT$ and this defines $w\in \GR_\cG(U)$.  We define the Schubert cell $C_w:=L^+\cG_{\mathring{U}}w \subset \GR_{\cG_U}$ and its closure $S_w:=\overline{C_w}\subset\GR_{\cG_U}$.
We see that any fibre of $S_w$ over $\mathring{U}$ is the Schubert variety of the fibre. Since the inductive limit of the fibrewise Schubert varieties exhausts the fibres of $\GR_\cG$ by \cite{Pappas-Rapoport} (and our $w$ form a cofinal system), we find that $\varinjlim S_w=\GR_{\cG_U}$. Furthermore, the canonical section of $\GR_\cG\to C$ given by the trivial $\cG$-bundle on $C$ factors through all $S_w$. Finally, since $U$ is a smooth curve, the projection $\pi_w\colon S_w\to U$ is flat.


We claim that this suffices to prove that the relative Picard functor $\Pic(\GR_\cG/C)$ is an \'etale sheaf (we adapt the arguments of \cite{BLR}, section 8.1):
By definition a line-bundle on an ind-scheme $\varinjlim Z_i$ is a family of line bundles $\cL_i$ on each of the $Z_i$, together with isomorphisms of the restrictions $\cL_i|_{Z_j} \cong \cL_j$ for all $j\leq j$. 

We use the $S_w$, which are flat over $U$. In particular for each $S_w$ we have an exact sequence:
$$H^1(U,\pi_{w,*}\bG_m) \to H^1(S_w,\bG_m) \to \Pic(S_w/U)(U) \to H^2(U,\pi_{w,*}\bG_m) \to H^2(S_w,\bG_m)$$
and the same holds for every base-change $T\to U$. Furthermore, if $\pi_{w,*}\cO_{S_w}=\cO_{S_w}$ then $\pi_{w,*}\bG_m=\bG_m$ so that the existence of our section $U\to S_w$ implies that the right hand arrow of the exact sequence is injective.

In our situation we know that for any $x\in U$ the fibre $(S_w)_x\subset \GR_{\cG,x}$ is a closed subscheme of finite type, so it will be contained in a connected, reduced Schubert variety $S_x$ of the fibre $\GR_{\cG_x}$. And $S_x$ will in turn be contained in $(S_{w^\prime})_x$ for some $w^\prime >w$. Therefore, the restriction map $\pi_{w^\prime,*} \cO_{S_{w^\prime}} \to \pi_{w,*}\cO_{S_w}$ will factor through $\cO_U$ and the same argument holds after any base-change $T\to U$.

This implies that given a compatible family of sections $(s_w)\in\Pic(S_w/U)(T)$ on some $T\to U$ we can lift these canonically to a line bundle on $\GR_{\cG}\times_U T$: First, let us check that each $s_w$ lies in the image of $H^1(S_{w,T},\bG_m)$. Given $w$ choose $w^\prime >w$ as in the previous paragraph. Then $s_w=s_{w^\prime}|_{S_{w,T}}$ and the obstruction to lift $s_w$ to $H^1(S_{w,T},\bG_m)$ lies in the image $H^2(T,\pi_{w^\prime,*} \bG_m) \to H^2(T,\bG_m) \to H^2(T,\pi_{w,*}\bG_m)$.
However, $H^2(T,\bG_m)\subset H^2(S_{w,T},\bG_m)$ so the obstruction must vanish and we can find line bundles $\cL_w^0$ on $S_{w,T}$, trivialized along our section of $S_w$, mapping to $s_w\in \Pic(S_w/U)(T)$.
Define $\cL_w:= \cL_{w^\prime}^0|_{S_{w,T}}$. Then for $v>w$ the bundles $\cL_v|_{S_{w,T}}$ and $\cL_w$ differ by an element of $H^1(T,\pi_{w,*}\bG_m)$, but as before this element lies in the image of $H^1(T,\bG_m)$, so it has to be trivial, because both bundles are trivialized along our section.
This procedure gives for any family of sections $(s_w)$ canonical preimages in $H^1((S_w)_T,\bG_m)$. Thus we find that $\Pic(\GR_\cG/C)$ is an \'etale sheaf. Fiberwise we can apply \cite{Pappas-Rapoport} 10.1 to see that $\Pic(\GR_\cG/C)/\prod X^*(\cG_x)$ is isomorphic to $\bZ$.

Finally, the formation of $\GR$ commutes with \'etale base change. Over $U:=C-\Ram(\cG)$ we can find a covering $\pi\colon \tilde{U}\to U$ such that the restriction $\pi^*(\cG)$ is an inner form. Therefore $\pi^*(c)$ is the constant sheaf $\bZ$, with canonical generator given by the ample line bundle of central charge $1$. In particular $c|_{U}$ is constant.
\end{proof}

\begin{remark}
\begin{enumerate}
 \item Denote by $s_0\colon C \to \GR_\cG$ the zero section. Then the composition $C \map{s_0} \GR_\cG \to \Bun_\cG$ is given by the trivial $\cG$-torsor, so we get a morphism $\Pic(\Bun_\cG)\to \Pic(\GR/C)$ of the groups of rigidified line bundles and thus a morphism $\Pic(\Bun_{\cG})\to H^0(C,c)$. Furthermore the last part of the lemma shows that for any point $x\in C-\Ram(\cG)$ the composition $\Pic(\Bun_{\cG})\to H^0(C,c) \to \Pic(\Gr_{\cG,x}) \map{c_x} \bZ$ does not depend on the chosen point $x$.
 \item Fix a point $x\in C-\Ram(\cG)$. Then we may find a bundle $\cL\in \Pic(\Bun_\cG)$ such that its image in $\Pic(\Gr_{\cG,x})=\bZ$ is non-zero as follows: Choose a faithful representation $\rho\colon \cG \to \SL(\cE)$, where $\cE$ is a vector bundle on $C$. This induces a morphism $\text{ind}_\rho\colon  \Bun_\cG \to \Bun_{\SL(\cE)}$. On $\Bun_{\SL(\cE)}$ we have the line bundle $\cL_{\det}$ given by the determinant of the cohomology of the universal vector bundle on $\Bun_{\SL(\cE)}\times C$. Thus we may define $\cL:=\text{ind}_\rho^*(\cL_{\det})$. It is known (\cite{Faltings_Loopgroups} p.42, \cite{BeauvilleLaszloConformalBlocks} p. 410, last two paragraphs) that the pull-back of $\cL_{\det}^{-1}$ to $\Gr_{\SL_n}=\Gr_{\SL(\cE),x}$ is an ample line bundle. Since $\rho$ induces an embedding $\Gr_{\cG,x} \to \Gr_{\SL_n,x}$ we see that the pull back of $\cL^{-1}$ to $\Gr_{\cG,x}$ is ample and in particular non trivial.

Together with the previous remark we find that there exists a minimal $n\in\bZ_{>0}$ such that the section $n\in H^0(C-\Ram(\cG),c)=\bZ$ extends to a global section i.e., a non-zero element in $H^0(C,c)$.
\item  Let $x\in \Ram(\cG)$, denote by $U_x:=C-\Ram(\cG)\cup \{x\} \subset C$ and $j_x\colon C-\Ram(\cG) \to U_x$ the open embedding. We claim that there exists $n_x\in \bZ_{>0}$ such that $c|_{U_x}\cong j_{x,*}(n_x\bZ) \coprod j_{x,!}(\bZ-\{n_x\bZ\})$ as a sheaf of sets: We already know that there exists a minimal nonzero section of $c|_{C-\Ram(\cG)}$, which extends to $U_x$. This defines an embedding $j_{x,*}(n_x\bZ) \coprod j_{x,!}(\bZ-\{n_x\bZ\}) \to c|_{U_x}$. Since the fiber of $c|_x\cong \bZ$ this must be an isomorphism. 
\item  The even unitary groups provide examples of groups for which the sheaf $c$ is not constant: Assume that $\text{char}(k)\neq 2$ and  choose a $\bZ/2\bZ$-covering $\tilde{C} \to C$, ramified at a nonempty set $\Ram\subset C$. Let $\cG=\SU_{\tilde{C}/C}(2n):=(\text{Res}_{\tilde{C}/C})(\SL_{2n})^\sigma\subset \text{Res}_{\tilde{C}/C} \SL_{2n}=:\cH$ be the corresponding unitary group. Pappas and Rapoport show (\cite{Pappas-Rapoport} Section 10.4) that for every $x\in \Ram(\cG)$ the induced map $\Gr_{\cG,x} \to \Gr_{\cH,x}$ defines an isomorphism of Picard groups. 

 If $y\in C-\Ram(\cG)$ is an unramified point, then over the formal completion $\widehat{\cO}_{C,y}$ we have $L\cH_y=L\SL_{2n} \times L\SL_{2n}$ and the action of $\bZ/2\bZ$ on $L\cH_y$ is given by permuting the factors and applying the transpose-inverse automorphism. In particular we find that $L\cG_y=L\SL_{2n}$, which is embedded as $A\mapsto (A,A^{t,{-1}})$ in $L\cH_y$. Therefore the corresponding map on Picard-groups is given by the sum $\Pic(\Gr_{\SL_{2n} \times \SL_{2n}}) = \bZ\times \bZ \to \bZ = \Pic(\Gr_{\SL_{2n}})$. 

For $x\in \Ram(\cG)$ we know by \cite{Pappas-Rapoport} Section 10.a.1 that the generator of $\Pic(\Gr_{\cH,x})$ is given by the determinant of the cohomology $\cL_{\det}$ of the corresponding universal vector bundle on $\Bun_{\tilde{C},\SL_n}=\Bun_\cH$.  For $y\in C-\Ram(\cG)$ the bundle $\cL_{\det}$ restricts to the diagonal element $(1,1)\in \Pic(\Gr_{\SL_{2n}\times \SL_{2n}})$. Thus we see that only twice the generator of the Picard group of $\Gr_{\cG,y}$ descends to $\Bun_{\cG}$.
\end{enumerate}
\end{remark}

\begin{theorem}\label{PicMG}
Assume that $\cG$ is simply connected, absolutely simple and splits over a tamely ramified extension of $k(C)$. Then there is an exact sequence:
$$ 0 \to \prod_{x\in \Ram(\cG)} X^*(\cG_x) \to \Pic(\Bun_\cG) \map{} \bZ \to 0.$$
\end{theorem}
To prove this result we need the analog of the loop group $L\cG$ for the sections of $\cG$ over open subsets $U\subset C$:
\begin{lemma}
Let $\cX\to C$ be a closed subscheme of $\bA^N\times C$, where $N$ is any integer. For any non-empty open $U\subset C$ the fpqc-sheaf 
$$ L^{out}_{U}\cX:= \cX(T\times U)$$
is an ind-scheme. The same holds for the limit over all $U\subset C$: $$L^{out}_{k(C)}\cX:=\varinjlim_{U\subset C} L^{out}_U\cG.$$
\end{lemma}
\begin{proof}
Let us first prove the lemma for the case $\cX=\bA^N_C$. Here the functor is given as $L^{out}_{U}\cX(T)=\bA^N(T\times U)$. Write $\cO_{C}(U)=\cup_n V_n$ as a union of finite dimensional $k$-vector spaces $V_n$. To give a map $T\times U\to \bA^N$ is the same as to give $N$ functions on $T\times U$, i.e., a finite $\cO_T(T)$-linear combination of elements in some $V_n$. This means that $L^{out}_{U}\cX$ is an inductive limit of affine spaces. The same argument holds for $L^{out}_{k(C)}\cX$.

Next, assume that $\cX\subset C\times \bA^N$ is a closed subscheme. Then $L^{out}_{U}\cX$ is a subfunctor of $L^{out}_{U}\bA_C^N$. To find equations for this subfunctor assume that $U\subsetneq C$ (otherwise the result is easy). Then $\cX|_{U}$ is defined by an ideal $I\subset H^0(\bA^n\times U,\cO_{\bA^n\times U})$ and we can choose generators $I=(f_1,\dots,f_m)$. An element $L^{out}_{U}\bA^N_C(T)$ defines an element of $L^{out}_{U}\cX(T)$ if and only if all the $f_i$ are mapped to $0$ in $\cO_T(T)\tensor \cO_C(U)$.
Writing $f_i$ as polynomials in the coordinates of $\bA^N_U$ the image in $\cO_T(T)\tensor \cO_C(U)$ can be expanded as an element in $\cO_T(T)\tensor V_n$ for large enough $n$, so choosing a basis of $V_n$ we obtain equations for the functor $L^{out}_{U}\cX$.
\end{proof}

\begin{proof}[Proof of Theorem \ref{PicMG}] (See \cite{LaszloSorger} for the case of constant groups.) 
First note that for any finite set of points $\{x_i\} \subset  C$ we have $\Pic(\prod_i \Gr_{\cG,x_i})=\prod_i \Pic(\Gr_{\cG,x_i})$, by the  see-saw principle, which we may apply since $\Gr_{\cG,x}$ is ind-proper and $\Pic(\Gr_{\cG,x})$ is discrete. Further, the product $\prod_i \Gr_{\cG,x_i}$ classifies $\cG$-torsors trivialized on $C-\{x_i\}$ and by the uniformization theorem we know that the $\Bun_\cG$ is the quotient of $\prod_i \Gr_{\cG,x_i}$ by the group $L^{out}_{C-\{x_i\}_i}\cG$, considered as a sheaf in the flat topology.

Now given $\cL\in \Pic(\Bun_\cG)$ we may tensor $\cL$ with a multiple fixed line bundle of minimal central charge and then modify it by a line bundle in the image of $\prod_{x\in C} X^*(\cG_x)$ such that the inverse image of $\cL$ to $\prod \Gr_{\cG,x_i}$ is trivial for any finite set of points $x_i\subset C$. Taking the limit over all points of $C$, we find that $\cL$ is defined by a character of the group $L^{out}_{k(C)}\cG$. Now, since $k$ is algebraically closed, $\cG_{k(C)}$ is quasi-split and simply connected. In particular $\cG(k(C))$ is generated by unipotent root subgroups $U_a(k(C))$. Given the structure of the $U_a$ (they are products of additive groups, or subgroups thereof, \cite{BT2}, 4.1) any point $u\in U_a(k(C))$ defines a morphism $\bG_a \to L^{out}_{k(C)}U_a$ considered as ind-schemes over $k$. Since $\bG_a$ does not have characters, the character of $L^{out}_{k(C)}$ must therefore be trivial on geometric points. This implies that the character is trivial on $L^{out}_{k(C)}\cG(T)$ for all reduced schemes $T$. Since we have seen in Theorem \ref{EinfachzusammenhaengendeUniformisierung} that there exists a smooth atlas $X\to \Bun_G$ which lifts to $\prod_i \Gr_{\cG,x_i}$ this is sufficient to show that $\cL$ must be trivial.
\end{proof}

\section{Existence of generic Borel subgroups}

In this final section we want to generalize the result of Drinfeld and Simpson on the existence of generic reductions to Borel subgroups. In the case of constant group schemes this was used to reduce to problems for tori e.g., Faltings used this result to construct line bundles of central charge one on $\Bun_G$ (see \cite{Faltings_Loopgroups}). As a first application in the more general setting of Bruhat-Tits group schemes this gives an alternative approach to the uniformization theorem.

In this section we allow non simply connected groups, however we need the assumption that $\cG$ splits over a tamely ramified extension.

Since some of the fibres of $\cG$ may not be semisimple, we begin by studying the possible extensions of Borel subgroups defined over the generic point of $C$ to these fibres. For us the following lemma, which we will prove by using arguments of \cite{BT2}, will suffice:
\begin{lemma}\label{ClosureBorel}
Assume that $\cG_{K_x}$ splits over a tamely ramified extension of $K_x$.
Denote by $\cT_0\subset \cG_{K_x}$ the maximal torus used in the definition of the Bruhat-Tits group scheme $\cG_{K_x}$ and let $\cB_{K_x} \subset \cG_{K_x}$ be a Borel subgroup with unipotent radical $\cU_{K_x}$.

Denote by $\cU\subset \cB \subset \cG$ be the closures of $\cU_{K_x},\cB_{K_x}$ in $\cG$. Then $\cB$ is smooth and $\cB/\cU\cong \cT_0$.
\end{lemma}
\begin{proof}
First, we may assume that the ground field is separably closed, since smoothness can be checked over $\overline{k}$.

{\em Reduction to the case of simply connected groups.} Let $\tilde{\cG}_{K_x}\to \cG_{K_x}$ be the simply connected covering of $\cG$. As noted in Section 4, the Bruhat-Tits buildings of $\cG$ and $\cG^\prime$ are canonically isomorphic (\cite{Landvogt}, 2.1.7). In particular this isomorphism defines the simply connected cover $\tilde{\cG}\to \cG$, an extension $\cZ \to \tilde{\cT}_0 \to \cT_0$ and $\tilde{\cB} \to \cB$. Since $\cT_0$ and $\tilde{\cT}_0$ are smooth and the unipotent radicals of $\cB$ and $\tilde{\cB}$ are isomorphic it is sufficient to prove the theorem for the simply connected group $\tilde{\cG}$.

{\em Reduction to the case of split groups.} Again, we can use the reduction given in \cite{Pappas-Rapoport} Section 7 and 8.e.2: $\cG_{K_x}\cong \prod \text{Res}_{K_i/K_x} \cG_i$ is a product where $K_i/K_x$ are tame extensions and $\cG_i$ is absolutely simple and simply connected. Thus we may assume that $K_x=K_i$ and $\cG_{K_x}=\cG_i$ is absolutely simple. In this case there is a tamely ramified extension $L/K_x$ with ring of integers $\cO_L$ such that $\cG_{\cO_x} \cong (\text{Res}_{\cO_L/\cO_x} G_{\cO_L})^\sigma$, where $G_{\cO_L}$ is a parahoric subgroup of the split Chevalley group scheme of the type given by $\cG_{K_x}$ and $\sigma$ is an automorphism of $\text{Res}_{\cO_L/\cO_x}G_{\cO_L}$. Since taking invariants preserves smoothness if the extension is tame (\cite{Edixhoven_Neron}), we may assume that $\cG_{K_x}=G_L$ is a split group.

Thus we are reduced to the following situation:
Denote by $G$ the Chevalley group scheme over $\cO_x$ with generic fiber $G_{K_x}$ and let $T\subset G$ be a split maximal torus and let $X^*(T)$ be its character group. We know that $\cG_{\cO_x}$ is a parahoric group scheme corresponding to a facet $A\subset X^*(T)$. Furthermore we are given a Borel subgroup $\cB_\eta \subset \cG_{K_x}=G_{K_x}$ and we want to show that the closure $\cB\subset \cG_{\cO_x}$ of $\cB_\eta$ is smooth.

Now let $P_A\subset G$ be the parabolic subgroup defined by $A$, let $B\subset P_A$ be a Borel subgroup and let $B^\prime\subset G$ be the closure of $\cB_\eta$ in $G$. Since all Borel subgroups are conjugate and $G/B$ is projective over $\cO_x$ we see that $B^\prime\subset G$ is of the form $B^\prime= g^{-1}Bg$ for some $g\in G(\cO)$. Denote by $T^g:=g^{-1}T g\subset B^\prime$ the corresponding maximal torus.

Denote the special fibres of $P_A$ and $B^\prime$ by $P_{A,x}$ and $B^\prime_x$. There exist a split maximal torus $T^\prime_x\subset B_x^\prime\cap P_{A,x}$. Since all tori are conjugate this implies that $T^\prime_x=b_x^{-1} T^g b_x$ for some $b_x\in B^\prime_x(k)$. Choose a lift $b\in B^\prime(\cO_x)$ of $b_x$. Then $T^\prime:=b^{-1}T^gb\subset B^\prime \subset G$ is a split maximal torus in $G$ the special fiber of which lies in $P_{A,x}$, in particular $T^\prime(\cO_x)\subset \cG(\cO_x)$ and therefore $T^\prime\subset \cG$. Thus $T^\prime \subset \cB$ and since the unipotent radical of $\cB_\eta$ also has a smooth extension to $\cG$ we see that $\cB$ is smooth.
\end{proof}
\begin{definition}
We will say that $\cB\subset \cG$ is a {\em Borel subgroup} of $\cG$ if $\cB$ is the closure of a Borel subgroup of the generic fiber of $\cG$.
\end{definition}
\begin{lemma}\label{Punktweise Reduktion auf Borel}
Assume that $k$ is algebraically closed. Let $\cB\subset \cG$ be a Borel subgroup and $\cP$ a $\cG$-torsor on $C$. Then there exists a reduction of $\cP$ to $\cB$.
\end{lemma}
\begin{proof}
Choose a point $x\in C-\Ram(\cG)$. By the uniformization theorem (Theorem \ref{Uniformisierung}) the restriction of $\cP$ to $C-\{x\}$ is trivial, in particular there exist a section $\mathring{s}$ of $\cP/\cB|_{C-\{x\}}$.  Now, since $\cG|_{C-\Ram(\cG)}$ is semisimple the quotient $\cP/\cB|_{C-\Ram(\cG)}$ is projective. Therefore the section $\mathring{s}$ extends to a section $s$ of $\cP/\cB$.  This proves the lemma.
\end{proof}
\begin{remark}
The only problem in the above construction stems from the fact that $\cG/\cB$ is  non-compact if $\cG_x$ is not semi-simple, and therefore sections of $\cP_\eta/\cB_\eta$ need not  extend to $\cP/\cB$. However, if one replaces $\cG$ by the inner form $\Aut_\cG(\cP)$ over $C$, one can apply 
Lemma \ref{ClosureBorel} to find that any reduction over $\cP_\eta$ extends to a Borel subgroup of $\Aut_\cG(\cP)$.

Note further that we used the uniformization theorem in the above proof. However, it would be sufficient to use the weaker statement, that there is an open subset containing the finitely many points $x\in \Ram(\cG)$ on which $\cP$ can be trivialized, which is easier to prove.
\end{remark}

The argument for the proof of the following proposition is a simple special case of the argument given by de Jong and Starr to produce sections of rationally connected fibrations (\cite{deJongStarr}):
\begin{proposition}\label{Deformation B}
Assume that $C$ and $\cG$ are defined over a field $k$. Let $\cB\subset \cG$ be a Borel subgroup, $\cP$ a $\cG$-torsor on $C$ and $s\in \cP/\cB$ a reduction of $\cP$ to $\cB$. Denote by $\cP_\cB$ the corresponding $\cB$-torsor on $C$.
\begin{enumerate}
\item If $H^1(C,\cP_{\cB}\times^{\cG} \Lie(\cG)/\Lie(\cB))=0$ then the map $\text{\rm ind}_\cB^\cG\colon \Bun_\cB \to \Bun_\cG$ which maps a $\cB$-torsor to the induced $\cG$-torsor is smooth in $\cP_\cB$.
\item If $s,\cB$ are given, then there exists another reduction $s^\prime\in \cP/\cB(k)$ such that (1) holds for $s^\prime$.\label{frei}
\end{enumerate}
\end{proposition}
\begin{corollary}[Reduction to generic Borel subgroups]\label{Reduktion_auf_Borel}
Let $\cG$ be a parahoric Bruhat-Tits group scheme over $C$ such that the generic fiber of $\cG$ is semisimple and quasi-split and let $\cB\subset \cG$ be a Borel subgroup. Then for any locally noetherian scheme $S$ and every family $\cP \in \Bun_\cG(S)$ there exists a smooth covering $S^\prime \to S$ such that $\cP|C\times S^\prime$ has a reduction to $\cB$.
\end{corollary}
\begin{proof}[Proof of Corollary \ref{Reduktion_auf_Borel}]
For every point $s\in S$ we can apply Corollary \ref{Punktweise Reduktion auf Borel} to obtain a  reduction $\cP_{s,\cB}$ of $\cP_s$ to $\cB$. By the preceding proposition we may further assume that the space of reductions of $\cP$ to $\cB_s$ is smooth in $\cP_{s,\cB_s}$ i.e., there is a smooth neighbourhood $S^\prime_s$ such that the reduction extends on this neighbourhood.
\end{proof}

\begin{proof}[Proof of Proposition \ref{Deformation B}]
The first part follows from the cohomology sequence:
$$ H^1(C,\cP_{\cB}\times^{\cB} \Lie{\cB}) \to H^1(C,\cP\times^{\cG}\Lie{\cG}) \to H^1(C,\cP_\cB\times^{\cB} \Lie(\cG)/\Lie(\cB))$$
and the fact that the first two groups classify infinitesimal deformations of the $\cB$- resp. $\cG$-bundle $\cP$.

To show the second part note that $\cP_\cB\times^{\cB}\Lie(\cG)/\Lie(\cB))$ is the normal bundle to the section $s$. In particular the morphism $\text{\rm ind}_\cB^\cG$ is smooth if the section $s$ is a very free curve, i.e. if the normal bundle has no higher cohomology. Now over all points $x\in C$ for which $\cG_x$ is semisimple, the fibres of $\cP/\cB \to C$ are flag varieties - in particular these contain very free rational curves and the smoothening argument for combs \cite{Kollar} II, Theorem 7.9 applies: Let $x_1,\dots x_n\in C$ be the set of points for which $\cG_x$ is not semi simple. There exist points $y_1,\dots,y_N\subset C$ and very free rational curves $P_i\subset (\cP/\cB)_{y_i}$ passing through $s(y_i)$. Consider $C^\prime:= s(C)\cup \bigcup_{i=1}^N P_i$. Then there exist a deformation $C^\pprime\subset \cP/\cB$ of $C^\prime$ such that $C^\pprime$ is smooth $C^\pprime_{x_i}=s(x_i)$ for all $i=1,\dots,n$ and such that the normal bundle of $C^\pprime$ has no higher cohomology. Since the degree of $C^\pprime$ over $C$ is still $1$ the curve $C^\pprime$ defines a new section $s^\prime$ which satisfies  condition (\ref{frei}).
\end{proof}

\begin{remark}
Using the strategy of Drinfeld and Simpson one can also use the above result to give a different proof of the uniformization theorem for groups splitting over a tamely ramified extension.
\end{remark}

\end{document}